\documentclass[11pt]{amsart}

\usepackage{adjustbox}
\usepackage{amssymb}
\usepackage{array}
\usepackage{multirow}
\usepackage{bm}
\usepackage{calc}
\usepackage{enumitem} 
\usepackage{caption} %makes space for captions wider
\usepackage{pbox} %needed for \pbox

\usepackage{longtable} 
\usepackage{arydshln} %needed for \hdashline and \cdashline in tables

\usepackage{wasysym} %needed for \diameter

\newcommand{\nc}{\newcommand}
\newcommand{\rnc}{\renewcommand}

\usepackage{hyperref}
\hypersetup{%
    pdfborder = {0 0 0},
    colorlinks,
    citecolor=blue!50!black,
    linkcolor=blue!50!black,
    filecolor=Darkgreen,
    citecolor={blue!50!black},
    urlcolor={blue!50!black}
}
\usepackage{bookmark} %to be used in conjunction with hyperref, to avoid the need for repeated compilations

\usepackage{tikz}
\usetikzlibrary{matrix,arrows,shapes,snakes,decorations.markings}
\nc{\tp}[2][baseline=-4pt,line width=.7pt,scale=75/100]{\trimbox{0pt -2pt 0pt -2pt}{\begin{tikzpicture}[{#1}] {#2} \end{tikzpicture}}}

\allowdisplaybreaks[1]

\numberwithin{equation}{section}

\newtheorem{thrm}{Theorem}[section]
\newtheorem{prop}[thrm]{Proposition}
\newtheorem{crl}[thrm]{Corollary}
\newtheorem{lemma}[thrm]{Lemma}
\newtheorem{conj}[thrm]{Conjecture}

\theoremstyle{definition}
\newtheorem{defn}[thrm]{Definition}
\newtheorem{exam}[thrm]{Example}

\theoremstyle{remark}
\newtheorem{rmk}[thrm]{Remark}

\nc{\rmkend}{\ensuremath{\diameter}}
\nc{\examend}{\ensuremath{\diameter}}
\nc{\defnend}{\ensuremath{\diameter}}

\setlist[itemize,1]{leftmargin=.4in}
\setlist[enumerate,1]{leftmargin=.4in,label=(\roman*)}
\setlist[description,1]{leftmargin=.4in,font=\normalfont\itshape}

\nc{\al}{\alpha}
\nc{\be}{\beta}
\nc{\ga}{\gamma}
\nc{\del}{\delta}
\nc{\eps}{\epsilon}
\nc{\veps}{\varepsilon}
\nc{\ze}{\zeta}
\nc{\ka}{\kappa}
\nc{\la}{\lambda}
\nc{\si}{\sigma}
\nc{\ups}{\upsilon}
\nc{\vphi}{\varphi}
\nc{\om}{\omega}

\nc{\Ga}{\Gamma}
\nc{\La}{\Lambda}
\nc{\Si}{\Sigma}
\nc{\Ups}{\Upsilon}

\nc{\ad}{{\sf ad}}
\nc{\cork}{{\sf corank}}
\rnc{\d}{{\sf d}}
\rnc{\exp}{{\sf exp}}
\nc{\id}{{\sf id}}
\nc{\im}{{\sf im}}
\rnc{\mod}{{\sf mod}}
\nc{\oi}{{\sf oi}}
\nc{\rk}{{\sf rank}}

\nc{\Ad}{{\sf Ad}}
\nc{\Aut}{{\sf Aut}}
\nc{\Br}{{\sf Br}}
\nc{\End}{{\sf End}}
\nc{\GL}{{\sf GL}}
\nc{\Hom}{{\sf Hom}}
\nc{\Inn}{{\sf Inn}}
\nc{\Ker}{{\sf Ker}}
\nc{\Out}{{\sf Out}}
\nc{\Sp}{{\sf Sp}}

\nc{\mfsl}{\mathfrak{s}\mathfrak{l}}

\nc{\mc}{\mathcal}
\nc{\mf}{\mathfrak}
\nc{\ms}{\mathsf}

\nc{\ol}{\overline}
\nc{\wh}{\widehat}
\nc{\wt}{\widetilde}

\nc{\ot}{\otimes}
\nc{\op}{\oplus}

\nc{\mfb}{\mathfrak{b}}
\nc{\mfc}{\mathfrak{c}}
\nc{\mfg}{\mathfrak{g}}
\nc{\mfh}{\mathfrak{h}}
\nc{\mfk}{\mathfrak{k}}
\nc{\mfp}{\mathfrak{p}}
\nc{\mfn}{\mathfrak{n}}
\nc{\mft}{\mathfrak{t}}

\nc{\N}{\mathbb{N}}
\nc{\Z}{\mathbb{Z}}
\nc{\Q}{\mathbb{Q}}
\nc{\R}{\mathbb{R}}
\nc{\C}{\mathbb{C}}
\nc{\F}{\mathbb{F}}

\nc{\eq}[1]{\begin{equation}#1\end{equation}}
\nc{\eqrefs}[2]{\text{(\ref{#1}-\ref{#2})}}

\nc{\qu}{\quad}
\nc{\qq}{\qquad}

\rnc{\,}{\kern 0.1em} % standard is \kern 0.16667em

\makeatletter
\nc*\rel@kern[1]{\kern#1\dimexpr\macc@kerna}
\nc*\widebar[1]{%
  \begingroup
  \def\mathaccent##1##2{%
    \rel@kern{0.8}%
    \overline{\rel@kern{-0.8}\macc@nucleus\rel@kern{0.2}}%
    \rel@kern{-0.2}%
  }%
  \macc@depth\@ne
  \let\math@bgroup\@empty \let\math@egroup\macc@set@skewchar
  \mathsurround\z@ \frozen@everymath{\mathgroup\macc@group\relax}%
  \macc@set@skewchar\relax
  \let\mathaccentV\macc@nested@a
  \macc@nested@a\relax111{#1}%
  \endgroup
}
\makeatother
\makeatletter
\def\nobreakhline{%
  \noalign{\ifnum0=`}\fi
    \penalty\@M
    \futurelet\@let@token\varLambda_{q,}@@nobreakhline}
\def\LT@@nobreakhline{%
  \ifx\@let@token\hline
    \global\let\@gtempa\@gobble
    \gdef\LT@sep{\penalty\@M\vskip\doublerulesep}% <-- change here
  \else
    \global\let\@gtempa\@empty
    \gdef\LT@sep{\penalty\@M\vskip-\arrayrulewidth}% <-- change here
  \fi
  \ifnum0=`{\fi}%
  \multispan\LT@cols
     \unskip\leaders\hrule\@height\arrayrulewidth\hfill\cr
  \noalign{\LT@sep}%
  \multispan\LT@cols
     \unskip\leaders\hrule\@height\arrayrulewidth\hfill\cr
  \noalign{\penalty\@M}%
  \@gtempa}
\makeatother

\makeatletter
\def\@cline#1-#2\@nil{%
  \omit
  \@multicnt#1%
  \advance\@multispan\m@ne
  \ifnum\@multicnt=\@ne\@firstofone{&\omit}\fi
  \@multicnt#2%
  \advance\@multicnt-#1%
  \advance\@multispan\@ne
  \leaders\hrule\@height\arrayrulewidth\hfill
  \cr
  \noalign{\nobreak\vskip-\arrayrulewidth}}
\makeatother

\nc{\red}{\color{red}}
\nc{\blu}{\color{blue}}
\nc{\br}{\color{brown}}
\nc{\grn}{\color{green!55!black}}
\nc{\gry}{\color{gray}}

\rnc\appendixname{}

\begin{document}

\title[Pseudo-symmetric pairs for Kac-Moody algebras]
{Pseudo-symmetric pairs for Kac-Moody algebras}

\author{Vidas Regelskis}
\address{Department of Physics, Astronomy and Mathematics, University of Hertfordshire, Hatfield AL10 9AB, UK, and
Institute of Theoretical Physics and Astronomy, Vilnius University, Saul\.etekio av.~3, Vilnius 10257, Lithuania.}
\email{vidas.regelskis@gmail.com}

\author{Bart Vlaar}
\address{Department of Mathematics, Heriot-Watt University, Edinburgh, EH14 4AS, UK, and Max Planck Institute for Mathematics, Vivatsgasse 7, 53111 Bonn, Germany}
\email{B.Vlaar@hw.ac.uk}

	
\maketitle

\begin{center}
\emph{To Jasper Stokman on the occasion of his 50th birthday}
\end{center}

\begin{abstract}
Lie algebra involutions and their fixed-point subalgebras give rise to symmetric spaces and real forms of complex Lie algebras, and are well-studied in the context of symmetrizable Kac-Moody algebras. 
In this paper we propose a generalization.
Namely, we introduce the concept of a pseudo-involution, an automorphism which is only required to act involutively on a stable Cartan subalgebra, and the concept of a pseudo-fixed-point subalgebra, a natural substitute for the fixed-point subalgebra.
In the symmetrizable Kac-Moody setting, we give a comprehensive discussion of pseudo-involutions of the second kind, the associated pseudo-fixed-point subalgebras, restricted root systems and Weyl groups, in terms of generalizations of Satake diagrams.
\end{abstract}

% TOC

\setcounter{tocdepth}{1} 
\tableofcontents
\setcounter{tocdepth}{2}

%%%%%%%%%%%%%%%%%%%%%%%%%%%%%%%%%%%%%%%%%%%%%%%%%%%%%%%%%%%%%%%%%%%%%
%  Section 1: Introduction
%%%%%%%%%%%%%%%%%%%%%%%%%%%%%%%%%%%%%%%%%%%%%%%%%%%%%%%%%%%%%%%%%%%%%

\section{Introduction} 
 
Given a set, possibly with additional structure, its involutive automorphisms and corresponding fixed-point subsets are natural objects of interest.
In the case of a Lie algebra\footnote{We always work over an algebraically closed field of characteristic 0. Lie algebra automorphisms are always understood to fix this field pointwise.} $\mfg$, an involutive automorphism $\theta: \mfg \to \mfg$ and its fixed-point subalgebra $\mfg^\theta$ give rise to the notion of a symmetric pair $(\mfg,\mfg^\theta)$ and a linear decomposition of $\mfg$ in terms of $\pm 1$-eigenspaces: $\mfg = \mfg^\theta \oplus \mfg^{-\theta}$.
These have been extensively studied for semisimple finite-dimensional Lie algebras, see e.g.\ \cite{Sa60,Ar62,He12}, for instance as part of the investigation of the associated symmetric space $G/G^\theta$ or the real form of $\mfg$ defined in terms of the semi-involution associated to $\theta$; key tools for this study are the restricted root system and the restricted Weyl group associated to $\theta$.
In these works the involutions are described combinatorially in terms of Satake diagrams, which are particular decorations of the Dynkin diagram associated to $\mfg$ (there are also other descriptions in terms of so-called Kac or Vogan diagrams, see e.g.\ \cite{Ka69,OV94,CH04,CZ17}).

Kac-Moody algebras, which we always assume to be symmetrizable and indecomposable, generalize simple finite-dimensional Lie algebras as well as their (suitably extended) loop algebras \cite{Ka68,Mo68,Ka74,PK83,Ka90}.
Involutive automorphisms for Kac-Moody algebras were studied for the affine case in \cite{Ba86,Le88} and more generally in \cite[Ch.~5]{KW92}.
In the infinite-dimensional case, a crucial distinction exists between those automorphisms $\theta$ which preserve the two conjugacy classes of Borel subalgebras and those which swap them; they are said to be \emph{of the first} and \emph{of the second kind} and the corresponding real forms are called \emph{almost compact} and \emph{almost split}, respectively.
Involutions of the second kind, which are naturally described by Kac-Moody versions of Satake diagrams, were studied in \cite{Be89} and, with a particular focus on the associated real forms, in \cite{BBBR95,BR06}. 
They were revisited in the context of q-deformations of their fixed-point subalgebras in \cite[Sec.\ 2]{Ko14}.
For involutions of the first kind, see e.g. \cite{Na92,BR03,BR07}.

Such classifications of involutions rely on detailed knowledge of the structure of the automorphism group in general, see e.g.\ \cite{PK83,Ba86,KW92,Gu10,CZ17}.
A natural generalization of the study of involutive automorphism and their fixed-point subalgebras is obtained when we replace ``involutive'' by ``finite-order'', see e.g.\ \cite{Ka69,JZh01,HG09}. 

\subsection{Pseudo-involutions and pseudo-fixed-point subalgebras} 

In this paper we pursue a rather different direction.
Namely, we consider the following two questions for an arbitrary Lie algebra $\mfg$, defined over an algebraically closed field $\F$ of characteristic 0, and a distinguished subalgebra $\mft \subseteq \mfg$:

\begin{enumerate}
\item 
Are there (non-involutive) automorphisms $\theta: \mfg \to \mfg$ which stabilize $\mft$ such that the restriction $\theta|_\mft$ is an involution?
Can we classify them? \smallskip
\item
What is a reasonable substitute $\mfk \subseteq \mfg$ for the fixed-point subalgebra $\mfg^\theta$? 
Since $\theta|_\mft$ is still an involution, it is natural to require that $\mfk \cap \mft = \mft^\theta$.
\end{enumerate}
Furthermore, it is natural to replace $\mft$, $\theta$ and $\mfk$ in this setup by their $\Aut(\mfg)$-conjugacy classes.\smallskip

To our best knowledge this problem has not been studied in great detail.
Of course, whether this is interesting depends on the choice of $\mft$; in particular, $\mft$ should not be too small.
If $\mft$ is a Cartan subalgebra%
\footnote{
In this paper by {Cartan subalgebra} we mean a maximal $\ad$-diagonalizable subalgebra of $\mfg$. 
Since $\F$ is algebraically closed, Cartan subalgebras are abelian. 
If in addition $\mfg$ is finite-dimensional its Cartan subalgebras are precisely its self-normalizing nilpotent subalgebras.
}, 
we believe that the problem above is of interest for many different $\mfg$, given the role that Cartan subalgebras play in representation theory.

Slightly modifying the above setup, we propose the following generalization of an involutive Lie algebra automorphism.

\begin{defn}  \label{def:pseudo-inv}
We call $\theta \in \Aut(\mfg)$ a \emph{pseudo-involution} if there exists a $\theta$-stable Cartan subalgebra $\mft \subseteq \mfg$ such that $\theta|_\mft$ is an involution, i.e.\ $\theta^2|_\mft = \id_\mft$.
\hfill \defnend
\end{defn}

Immediately we see that an involutive automorphism of $\mfg$ is a pseudo-involution if it stabilizes a Cartan subalgebra.
We expect that the resulting linear decomposition
\eq{
\mft = \mft^\theta \oplus \mft^{-\theta},
}
although weaker than the corresponding decomposition for $\mfg$, is strong enough to serve as a natural minimal condition for extensions of results in harmonic analysis on symmetric spaces to more general homogeneous spaces.

To address question (ii) we propose the following generalization of the notion of a symmetric pair (associated to $\theta$), which relies on the $\mft$-root space decomposition%
\footnote{
By $\mft$-root space we shall always mean a root space which is not associated to $0 \in \mft^*$, i.e.\ not $\mft$ itself. 
} 
of $\mfg$.
Denote by $\Phi^{(\mft)} \subset \mft^* \backslash \{0\}$ the $\mft$-root system and by $\mfg^{(\mft)}_\al \subset \mfg$ the $\mft$-root space associated to $\al \in \Phi^{(\mft)}$.

\begin{defn} \label{def:pseudofixedpoint}
Let $\theta \in \Aut(\mfg)$ be a pseudo-involution and let $\mft \subseteq \mfg$ be a $\theta$-stable Cartan subalgebra such that $\theta|_\mft$ is an involution. 
We call a subalgebra $\mfk \subseteq \mfg$ a \emph{pseudo-fixed-point subalgebra} if 
\begin{align}
\label{defn:psfixptsubalg:1} \mfk \cap \mft &= \mft^\theta, \\
\label{defn:psfixptsubalg:2} \dim \! \Big( \mfk \cap \big( \mfg_\al^{(\mft)} + \theta(\mfg_\al^{(\mft)}) \big) \Big) &= \dim(\mfg_\al^{(\mft)}) \qq \text{for all } \al \in \Phi^{(\mft)}.
\end{align}
We call the pair $(\mfg,\mfk)$ a \emph{pseudo-symmetric pair} (associated to $\theta$). \hfill \defnend
\end{defn}

The two conditions in Definition \ref{def:pseudofixedpoint} can be motivated by showing that they hold for fixed-point subalgebras of (suitable) involutions.
Indeed, let $\theta$ be an involutive automorphism of $\mfg$ with $\theta$-stable Cartan subalgebra $\mft$, so that $\theta(\mfg_\al^{(\mft)}) = \mfg_{\theta^*(\al)}^{(\mft)}$ for all $\al \in \Phi^{(\mft)}$, and set $\mfk = \mfg^\theta$.
We observe that \eqref{defn:psfixptsubalg:1} is trivially satisfied.
If $\theta^*(\al) \ne \al$ then the sum $\mfg_\al^{(\mft)} + \theta(\mfg_\al^{(\mft)})$ is direct and it is straightforward to show that a basis of $\mfg_\al^{(\mft)}$ can be $\theta$-symmetrized to obtain a basis of $\mfg_\al^{(\mft)} + \theta(\mfg_\al^{(\mft)})$.
If $\theta^*(\al) = \al$ then \eqref{defn:psfixptsubalg:2} is equivalent to $\mfg_\al^{(\mft)} \subseteq \mfg^\theta$ and hence one needs the additional assumption that $\theta$ fixes pointwise $\mfg_\al^{(\mft)}$, cf. \cite[5.15]{KW92}.
We obtain that, given an involutive automorphism $\theta$ with stable Cartan subalgebra $\mft$, the fixed-point subalgebra $\mfg^\theta$ is pseudo-fixed-point if $\theta$ fixes pointwise all stable $\mft$-root spaces, i.e. all $\mfg_\al^{(\mft)}$ such that $\al \in (\Phi^{(\mft)})^{\theta^*}$.
Observe that this condition is part of the definition of a \emph{maximally split} involution of $\mfg$, see e.g. \cite[Section 7]{Le02}.
\smallskip

In the rest of the paper we study pseudo-involutions of the second kind and their pseudo-fixed-point subalgebras of a Kac-Moody algebra $\mfg$. 
To make the representation-theoretic role for Cartan subalgebras more concrete we remark here that important categories of $\mfg$-modules are defined in terms of the standard Cartan subalgebra $\mfh$, notably the category $\mc{O}$ introduced by Kac \cite{Ka74} as a Kac-Moody generalization of the category introduced in \cite{Ge70,BGG71}.

Since every involution $\theta$ of the second kind of a Kac-Moody algebra has a stable Cartan subalgebra (see e.g.\ \cite[Lem.\ 5.7]{KW92}), involutions of the second kind are examples of pseudo-involutions.
Moreover, by \cite[Cor.\ 5.19]{KW92}, $\mfg^\theta$ is a pseudo-fixed-point subalgebra.\medskip

Recall that $\mfg$ is defined in terms of combinatorial datum (generalized Cartan matrix and associated Dynkin diagram). 
In terms of particular decorations of Dynkin diagrams, which we call \emph{compatible decorations}, see Definition \ref{def:cdec}, in Theorem \ref{thm:pseudoinv:class} we classify pseudo-involutions of $\mfg$ of the second kind.
It is intermediate between (1) a weaker statement about semisimple automorphisms of the second kind of $\mfg$ given in \cite[4.38-4.39]{KW92} in terms of more general combinatorial datum and (2) a classification of involutions of the second kind of $\mfg$, see \cite[Rmk.~5.33]{KW92}, \cite{BBBR95}, \cite[App.\ A]{Ko14} in terms of Satake diagrams (more special combinatorial datum).
The key ingredients of the proof of Theorem \ref{thm:pseudoinv:class} are basic decomposition results for $\Aut(\mfg)$, see e.g. \cite[Sec.~4]{KW92} and a handy result about twisted involutions in Coxeter groups due to Springer, see \cite{Sp85}.
We give a refinement in Corollary \ref{crl:pseudoinv:class:2} and point to a possible strengthening of this result in Conjecture \ref{conj:pseudoinv:class:2}.
Although these classifications follow straightforwardly from existing results on Kac-Moody algebras and Coxeter groups, we believe that due to the novelty of the concept of pseudo-involution, they qualify as main results of this paper.

If a compatible decoration satisfies an additional constraint then it is called a \emph{generalized Satake diagram}, see Definition \ref{def:GSat}.
Precisely in this case it gives rise to a pseudo-fixed-point subalgebra $\mfk$, see Theorem \ref{thm:k:GSat}, which is a direct Kac-Moody generalization of \cite[Thm.~3.3]{RV20}, written in the language of pseudo-fixed-point subalgebras.
From this an Iwasawa decomposition associated to the pair $(\mfg,\mfk)$ readily follows, see Corollary \ref{crl:Iwasawa}; such a statement was not included in \cite{RV20}.
Note that Iwasawa decompositions are well-known for symmetric spaces in the Kac-Moody setting, see e.g.\ \cite[Cor.~4(b)]{PK83} (written in terms of the Kac-Moody group $G$). 

The restricted Weyl group and restricted root system can be considered for any involutive automorphism of the root system of $\mfg$; in fact  the restricted Weyl group and root system can be studied without direct reference to $\mfg$, see e.g.\ \cite{Sch69,Lu76,He84,He91,Lu95} and also \cite[Ch.~25]{Lu03}, \cite[Sec.\ 2]{DK19} and \cite[Sec.\ 2]{RV20}.
Generalized Satake diagrams already arose in the work \cite{He84} on restricted root systems associated to involutions of finite root systems.
In this paper we consider the Kac-Moody generalization of this and show that generalized Satake diagrams unify various approaches to the restricted root system and Weyl group. 
The key results are Theorem \ref{thm:weylgroups}, Corollary \ref{WbarPhi:stable}, Theorem \ref{thm:Coxeter} and Theorem \ref{thm:Weylgroups2}.

\subsection{Applications in the quantum deformed setting}

If $\mfg$ is a symmetrizable Kac-Moody algebra, the universal enveloping algebra $U\mfg$ possesses a q-deformation $U_q\mfg$ (Drinfeld-Jimbo quantum group), see e.g.~\cite{Ji85,Dr87,Lu94}.
We briefly discuss the role of pseudo-involutions and pseudo-fixed-point subalgebras in the context of reflection equations (quartic braid relations) in representations of $U_q\mfg$.
Note that the canonical Cartan subalgebra $\mfh \subset \mfg$ plays an important role: since $U_q\mfh$ is generated by all invertible elements of $U_q\mfg$, only automorphisms of $\mfg$ which stabilize $\mfh$ can be lifted to algebra automorphisms of $U_q\mfg$.

Given an involutive automorphism $\theta$ of $\mfg$, consider the fixed-point subalgebra $\mfk = \mfg^\theta$.
For finite-dimensional $\mfg$ Letzter identified in \cite{Le99,Le02,Le03} a one-sided coideal subalgebra $U_q\mfk \subseteq U_q\mfg$ which q-deforms $U\mfk$ (cf.~\cite{NS95,NDS97} for alternative approaches to $U_q\mfk$).
A crucial step in Letzter's work was to ``arrange'' $\theta$ (by conjugacy) to make it maximally split: it stabilizes $\mfh$, fixes pointwise all $\theta$-stable positive root spaces and sends other positive root spaces to negative root spaces, see \cite[Eqns.\ (3.1)-(3.3)]{Le99} and \cite[Sec.\ 7]{Le02}.

In addition, there is a ${\sf B}_2$-analogue of quasitriangularity for these subalgebras: there exists an element $\mc K$ in the centralizer of $U_q\mfk$ (in a completion of $U_q\mfg$) which satisfies the type ${\sf B}_2$-braid relation.
The construction of this element was described by Balagovi\'{c} and Kolb in \cite{BK19}, building upon an earlier construction in \cite{BW18} due to Bao and Wang for a particular $\mfk$ in the case $\mfg = \mfsl_N$. 
These works place quantum symmetric pairs in the context of \emph{reflection equation algebras}, see \cite{KS92,KS93,KS09}.
In \cite{RV20} we showed that the above construction for $\mfg$, with minor adjustments, is applicable in a larger setting.
The Lie-theoretic underpinning of this is precisely given by pseudo-involutions and their pseudo-fixed-point subalgebras.

In \cite{Ko14}, Kolb provided the Kac-Moody generalization of Letzter's construction of $U_q\mfk$ of $\mfk = \mfg^\theta$ in the case that $\theta$ is an involution of the second kind, using a combinatorial approach to such $\theta$ indicated in \cite[Sec.~5]{KW92} and discussed in \cite{BBBR95}.
Recall that Kac-Moody algebras of affine type are (extensions of) loop algebras of finite-dimensional simple Lie algebras.
Indeed, in the affine case suitable interwiners for $U_q\mfk$ are known, see e.g.\ \cite{DM03,RV16}, to satisfy the \emph{reflection equation} (boundary Yang-Baxter equation) with spectral parameter \cite{Ch84,Sk88,Ch92}.
Therefore pseudo-fixed-point subalgebras are relevant to the study of (quantum) integrability in the presence of a boundary.

In \cite{AV20} universal versions of these interwiners were constructed by extending the approach from \cite{BK19} to the Kac-Moody setting.
Moreover, Kolb in \cite{Ko21} showed that the existence of the bar involution for $U_q\mfk$ follows from the construction in \cite{AV20} which does not rely on explicit presentations of $U_q\mfk$.
The works \cite{AV20,Ko21} deal with quantizations of pseudo-fixed-point subalgebras $\mfk$ of Kac-Moody algebras, not just fixed-point subalgebras of involutions.
Therefore the present paper provides a Lie-theoretic context for the q-deformed constructions in \cite{AV20,Ko21}.

\enlargethispage{1.5em} % to fit the next paragraph in one page

This work is also intended to serve as a platform supporting generalizations, in terms of a homogeneous framework for the entire range of pseudo-symmetric pairs, of advanced representation-theoretic applications, both undeformed and q-deformed.
Thus far these have mainly focused on quasi-split symmetric pairs; we mention \cite{SV15,RSV15a,RSV15b,RSV18}, \cite{SR20,RS20,St21} and \cite{ES18,BW18,LW20} as illustrative examples of various flavours.

\subsection{Outline}

This paper is organized as follows. 
Section \ref{sec:pseudoinvolutions} discusses pseudo-involutions in the Kac-Moody setting in terms of compatible decorations.
Subsections \ref{sec:GCM}-\ref{sec:Autg} are used to introduce background material, mainly drawn from \cite{KW92} and \cite{BBBR95}, and fix notation. 

In Section \ref{sec:pseudofixedpoint} we look at pseudo-fixed-point subalgebras of Kac-Moody algebras in terms of generalized Satake diagrams, thereby providing Kac-Moody generalizations of the finite-dimensional work in \cite[Sec.~2~and~3]{RV20}.

In Section \ref{sec:restricted} we survey, combine and extend various approaches to the corresponding restricted root system and restricted Weyl group, creating a comprehensive catalogue of results in this area.

Finally, in Appendix \ref{sec:tables} we provide a classification of generalized Satake diagrams of finite and affine types, combining and adding to existing lists, featuring a new easy-to-use notation.
We hope that these tables prove to be a useful reference for the community.\medskip 

In this paper we will follow (almost precisely) Carter's ``Dynkin notation'' for the affine root systems, see \cite[App.]{Ca05}.
We have the following correspondence with the customary notation due to Kac, see e.g. \cite[Tables Aff 1-3]{Ka90}:
\[
\begin{array}{llllllllllllllll}
\text{(Dynkin)} & \wh{\sf A}_{n \ge 1} & \wh{\sf B}_{n \ge 3} & \wh{\sf B}^\vee_{n \ge 3} & \wh{\sf C}_{n \ge 2} & \wh{\sf C}^\vee_{n \ge 2} & \wh{\sf C}'_{n \ge 1}  & \wh{\sf D}_{n \ge 4} & \wh{\sf E}_{6,7,8} & \wh{\sf F}_4 & \wh{\sf F}_4^\vee & \wh{\sf G}_2 & \wh{\sf G}_2^\vee
\\[4pt]
\text{(Kac)} & {\sf A}_n^{(1)} & {\sf B}_n^{(1)} & {\sf A}_{2n-1}^{(2)} & {\sf C}_n^{(1)} & {\sf D}_{n+1}^{(2)} & {\sf A}_{2n}^{(2)} & {\sf D}_n^{(1)} & {\sf E}_{6,7,8}^{(1)} & {\sf F}_4^{(1)} & {\sf E}_6^{(2)} & {\sf G}_2^{(1)} & {\sf D}_4^{(3)}
\end{array}
\]
where for the classical Lie types we have indicated the usual constraints on $n$ (the rank of the generalized Cartan matrix) to avoid low-rank coincidences.\medskip

{\it Acknowledgements.} The authors are grateful to the referees for valuable suggestions.
The second-named author was supported by the UK Engineering and Physical Sciences Research Council, grant number EP/R009465/1. 

%%%%%%%%%%%%%%%%%%%%%%%%%%%%%%%%%%%%%%%%%%%%%%%%%%%%%%%%%%%%%%%%%%%%%
%  EoF
%%%%%%%%%%%%%%%%%%%%%%%%%%%%%%%%%%%%%%%%%%%%%%%%%%%%%%%%%%%%%%%%%%%%%

%%%%%%%%%%%%%%%%%%%%%%%%%%%%%%%%%%%%%%%%%%%%%%%%%%%%%%%%%%%%%%%%%%
% Section 2
%%%%%%%%%%%%%%%%%%%%%%%%%%%%%%%%%%%%%%%%%%%%%%%%%%%%%%%%%%%%%%%%%%

\section{Pseudo-involutions in terms of compatible decorations} \label{sec:pseudoinvolutions}

This section gives a basic and, we hope, pedagogical account of the classification of pseudo-involutions of the second kind of Kac-Moody algebras in terms of decorations of Dynkin diagrams.
We begin by reviewing some theory of Kac-Moody algebras and their automorphisms, for which our main references are \cite{KW92} and \cite{BBBR95}.

%%% 2.1.

\subsection{Generalized Cartan matrices and Dynkin diagrams} \label{sec:GCM}

Given a finite index set $I$, let $A=(a_{ij})_{i,j \in I}$ be a generalized Cartan matrix, i.e.\ $a_{ii}=2$ for all $i \in I$, $a_{ij} \in \Z_{\le 0}$ for all distinct $i,j \in I$ and $a_{ij} =0$ if and only if $a_{ji}=0$.
We will always assume that $A$ is symmetrizable, i.e.\ there exist setwise-coprime positive integers $\eps_i$ ($i \in I$) such that $\eps_i a_{ij} = \eps_j a_{ji}$ for all $i,j \in I$.
The group of \emph{diagram automorphisms} of $A$ is the finite group
\eq{
\Aut(A) := \{ \tau: I \to I \text{ invertible} \, : \, a_{\tau(i)\tau(j)} = a_{ij} \text{ for all } i,j \in I \}.
}
Let $J \subseteq I$.
The principal submatrix  $A_J := (a_{ij})_{i,j \in J}$ is also a symmetrizable generalized Cartan matrix.
We set 
\eq{
J^\perp := \{ i \in I \, : \, a_{ij} =0 \text{ for all } j\in J\}.
}
We call $K \subseteq J$ a \emph{component} of $J$ if $J \subseteq K \cup K^\perp$.
We call $J$ \emph{connected} if it is nonempty and has no nonempty proper components.
From now on we assume $I$ is connected; hence $(\eps_i)_{i \in I}$ is uniquely determined by $A$.

We say that $A$ is \emph{of finite type} if $\det(A_J) > 0$ for all $J \subseteq I$ and \emph{of affine type} if $\det(A)=0$ and $\det(A_J) > 0$ for all $J \subset I$.
Recall that if $a_{ij}a_{ji} \le 4$ for all $i,j \in I$, which is satisfied if $A$ is of finite or affine type, the Dynkin diagram associated to $(I,A)$ is an oriented multigraph with vertex set $I$ with the edges determined as follows.
Distinct vertices $i,j \in I$ are connected by $\max(|a_{ij}|,|a_{ji}|)$ edges; if $\eps_i=\eps_j$ we do not assign an orientation to the edges between $i$ and $j$ and otherwise these edges point to the node with the smaller value of $\eps$.

%%% 2.2.

\subsection{Braid group and Weyl group} \label{sec:braidgroup}

Associated to $A$ is the Artin-Tits braid group
\eq{
\Br = \Big\langle \{ T_i \}_{i \in I} \, : \, \underbrace{T_iT_j \cdots}_{m_{ij}} = \underbrace{T_jT_i \cdots}_{m_{ij}} \qu \text{if } i \ne j \Big\rangle
}
where $m_{ij}=2,3,4,6$ if $a_{ij}a_{ji}=0,1,2,3$, respectively, and $m_{ij}=\infty$ if $a_{ij}a_{ji} \ge 4$.
The assignment $T_i \mapsto s_i$ extends to a surjective group map from $\Br$ to the \emph{Weyl group}
\eq{
W = \Big\langle \{ s_i \}_{i \in I} \, : \, s_i^2=1, \, (s_is_j)^{m_{ij}} = 1 \text{ if } i \ne j \Big\rangle.
}
The pair $(W,\{ s_i \}_{i \in I})$ is a Coxeter system 
and hence associated to it is the length function $\ell$ and the notion of reduced expression, see e.g. \cite[Ch. IV, \S 1.1]{Bo68}.
We define a map $T: W \to \Br$ by $T(s_{i_1}\cdots s_{i_\ell})=T_{i_1}\cdots T_{i_\ell}$ if $s_{i_1} \cdots s_{i_\ell}$ is reduced; hence $T(ww') = T(w)T(w')$ for all $w,w' \in W$ such that $\ell(ww') = \ell(w)+\ell(w')$.

The action of $\Aut(A)$ on $\Br$ defined by $\tau(T_i) = T_{\tau(i)}$ descends to an action of $\Aut(A)$ on $W$ which respects $\ell$ so that $\tau(T(w)) = T(\tau(w))$ for all $\tau \in \Aut(A)$, $w \in W$.

%%% 2.3.

\subsection{Minimal realization and bilinear forms} \label{sec:h}

Recall that $\F$ is an algebraically closed field of characteristic 0.
We fix a \emph{minimal realization} $(\mfh,\Pi,\Pi^\vee)$ of $A$ over $\F$; that is, $\mfh$ is a $(|I|+\cork(A))$-dimensional $\F$-linear space and $\Pi^\vee = \{h_i\}_{i \in I} \subset \mfh$ and $\Pi = \{\al_i\}_{i \in I} \subset \mfh^*$ are linearly independent subsets such that $\al_j(h_i)=a_{ij}$ for all $i,j \in I$.
There is a unique symmetric bilinear $\F$-valued form $(\phantom{x},\phantom{x})$ on $\mfh':= \Sp_\F \Pi^\vee$ with the property $(h_i,h_j) = \eps_j^{-1} a_{ij}$ for all $i,j \in I$.
By \cite[2.1]{Ka90} one may choose any complementary subspace $\mfh''$ to $\mfh'$ in $\mfh$ and extend $(\phantom{x},\phantom{x})$ to a nondegenerate symmetric bilinear $\F$-valued form on $\mfh$ with the properties
\eq{ \label{form:def}
(h,h_j) = \eps_j^{-1} \al_j(h), \qq (h',h'')=0 \qq \text{for all } h \in \mfh, \, j \in I \text{ and } h',h'' \in \mfh''.
}
Consider the linear map $\nu: \mfh \to \mfh^*$ defined by $\nu(h)(h') = (h,h')$ for all $h,h' \in \mfh$; it satisfies $\nu(h_i) = \eps_i^{-1} \al_i$ for all $i \in I$.
Since $(\phantom{x},\phantom{x})$ is nondegenerate, $\nu$ is an isomorphism.
We now define a symmetric bilinear form on $\mfh^*$ by $(\la,\mu)=(\nu^{-1}(\la),\nu^{-1}(\mu))$, which satisfies $(\al_i,\al_j) = \eps_i a_{ij}$ for all $i,j \in I$.
In the remainder of this paper, whenever we discuss orthogonality, self-adjointness or isometry, it will always be with respect to $(\phantom{x},\phantom{x})$.

There are faithful linear isometric actions of $W$ on $\mfh$ and $\mfh^*$ determined by
\eq{
s_i(h) = h-\al_i(h)h_i, \qq s_i(\al) = \al - \al(h)\al_i \qq \text{for } i \in I, \, h \in \mfh, \, \al \in \mfh^*.
}
The group $\Aut(A)$ acts via relabelling on $\mfh'$; this can be extended to an action on $\mfh$ by choosing $\mfh''$ according to \cite[4.19]{KW92}.
The isomorphism $\nu$ allows us to define an action of $\Aut(A)$ on $\mfh^*$ and the bilinear forms $(\phantom{x},\phantom{x})$ are $\Aut(A)$-invariant.

%%% 2.4.

\subsection{Kac-Moody algebra and roots}

Let $\mfg=\mfg(A)$ be the (indecomposable symmetrizable) Kac-Moody Lie algebra defined in terms of $A$ with Chevalley generators $e_i,f_i$ ($i \in I$) and $h_i:=[e_i,f_i]$, see \cite[1.1 and 1.2]{KW92}.
The derived subalgebra of $\mfg$ is the $\cork(A)$-codimensional subalgebra $\mfg'$ generated by all $e_i$ and $f_i$ ($i \in I$).
The centre $\mfc$ is contained in $\mfh' = \mfh \cap \mfg'$.
The action of $\Aut(A)$ on $\mfh$ extends to an action by Lie algebra automorphisms of $\mfg$ if we set $\tau(e_i)=e_{\tau(i)}$ and $\tau(f_i)=f_{\tau(i)}$ for all $\tau \in \Aut(A)$, $i \in I$.

We denote the subalgebra of $\mfg$ generated by all $e_i$ by $\mfn^+$ and the subalgebra generated by all $f_i$ by $\mfn^-$.
As identities of $\mfh$-modules, we have the triangular decomposition $\mfg = \mfn^+ \oplus \mfh \oplus \mfn^-$ holds and the root space decomposition
\eq{ 
\label{g:rootspacedecomposition}
\mfg = \bigoplus_{\al \in \mfh^*} \mfg_\al, \qq \mfg_\al = \{ x \in \mfg\, : \, \forall h \in \mfh,\, [h,x] = \al(h)x \}.
}

We call $\al \in \mfh^*$ a \emph{root} if $\mfg_\al \ne \{ 0 \}$ and define the root system $\Phi$ as the set of nonzero roots (note that $\mfg_0 = \mfh$). 
The \emph{root lattice} is $Q = \Z \Phi = \Z \Pi$. 
We denote the positive cone of $Q$ by $Q^+ = \Z_{\ge 0} \Pi$.
We set $\Phi^+ = \Phi \cap Q^+$ so that $\mfn^+ = \bigoplus_{\al \in \Phi^+} \mfg_\al$, $\Pi \subseteq \Phi^+$ and $\Phi = \Phi^+ \cup (-\Phi^+)$.

The action of $W$ on $\mfh^*$ stabilizes $\Phi$.
We call $\al \in \Phi$ \emph{real} if $\al \in W(\Pi)$ and otherwise \emph{imaginary}.
If $\al = w(\al_i)$ for some $w \in W$, $i \in I$ then $\al^\vee := w(h_i) \in \mfh$ is well-defined, $\Phi \cap \Z \al = \{ \al,-\al \}$, $(\al,\al)>0$, the root space $\mfg_\al$ is 1-dimensional and for all $x \in \mfg_\al$ the adjoint map $\ad(x): \mfg \to \mfg$ is a locally nilpotent derivation.

%%% 2.5.

\subsection{Kac-Moody group and triple exponentials} \label{sec:tripleexp}

Recall the Kac-Moody group $G$ and, for $\al \in W(\Pi)$, the map $\exp: \mfg_\al \to G$, see \cite[1.3]{KW92}.
There is a group morphism $\Ad: G \to \Aut(\mfg)$ uniquely determined by $\Ad(\exp(x)) = \exp(\ad(x))$ for all $x \in \mfg_\al$ with $\al \in \Phi$ real.
There is also a group morphism from $\Br$ to $G$ given by \emph{triple exponentials}
\eq{
T_i \mapsto n_i:= \exp(e_i)\,\exp(-f_i)\,\exp(e_i).
}
Denote by $N$ the subgroup of $G$ which is the image of this morphism; note that $\Aut(A)$ acts on $N$ by relabelling.
We compose the map $T:W \to \Br$ defined in Section \ref{sec:braidgroup} with this morphism to obtain a map $n:W \to G$ with image $N$ which satisfies $n(s_i)=n_i$ for all $i \in I$ and $n(ww')=n(w)\,n(w')$ if $w,w' \in W$ are such that $\ell(ww')=\ell(w)+\ell(w')$.
For all $i \in I$ and all $\tau \in \Aut(A)$ it follows that
\eq{
\Ad(n_i)|_\mfh = s_i, \qq \tau \circ \Ad(n_i) = \Ad(n_{\tau(i)}) \circ \tau. 
}

Let $\wt H:= \Hom(Q,\F^\times)$ denote the group of characters on the root lattice $Q$.
There is a group morphism $\Ad : \wt H \to \Aut(\mfg)$ given by
\eq{
\Ad(\chi)(x) = \chi(\al)\, x, \qq \chi \in \wt H, \, x \in \mfg_\al, \, \al \in \Phi,
}
which induces an action of $\wt H$ on $G$, so we may consider $\wt H \ltimes G$ and $\Ad(\wt H \ltimes G) < \Aut(\mfg)$, see \cite[1.10]{KW92}.
By \cite[Prop. 4.10.1]{BBBR95},
\eq{ \label{nw:inverse}
\Ad(n(w^{-1})) = \Ad(\ze(w)) \circ \Ad(n(w))^{-1} 
}
for all $w \in W$, where $\ze(w) \in \wt H$ is defined by 
\eq{
\ze(w)(\la) = \prod_{\al \in \Phi^+ \cap\hspace{0.075em} w(-\Phi^+)} (-1)^{\la(\al^\vee)}, \qq \la \in Q.
}

%%% 2.6.

\subsection{Subdiagrams of finite type} \label{sec:subfinite}

Given a subset $X \subseteq I$ we may consider the Lie subalgebra $\mfg_X := \langle \{ e_i,f_i \}_{i \in X} \rangle$, the Cartan subalgebra $\mfh_X = \mfh \cap \mfg_X$, the parabolic Weyl group $W_X := \langle \{ s_i \}_{i \in X} \rangle$ and the root subsystem $\Phi_X := \Phi \cap Q_X$ where $Q_X = \sum_{i \in X} \Z \al_i$.
The following statements are equivalent: $A_X$ is of finite type; the restriction of $(\phantom{x},\phantom{x})$ to $\mfh_X \times \mfh_X$ is positive definite; $W_X$ is finite; $\mfg_X$ is finite-dimensional; $\Phi_X$ is finite; all elements of $\Phi_X$ are real.
In this case $\mfg_X$ is semisimple and we simply say that $X$ is of finite type, which we assume henceforth.

The unique longest element $w_X \in W_X$ is necessarily an involution and there exists an involutive diagram automorphism $\oi_X \in \Aut(A_X)$, called \emph{opposition involution}, such that $w_X(\al_i) = -\al_{\oi_X(i)}$ for all $i \in X$.
Immediately we obtain
\eq{ \label{wX:conjugation}
w_X \cdot s_i = s_{\oi_X(i)} \cdot w_X \qq \text{for all } i \in X
}
in $W_X$.
The opposition involution of a connected subset $X \subseteq I$ of finite type is the trivial diagram automorphism except in the case where $X$ is of type ${\sf A}_r$ with $r>1$, ${\sf D}_r$ with $r >3$ odd or ${\sf E}_6$; in these cases $\oi_X$ is the unique nontrivial diagram automorphism.

Because any $\tau \in \Aut(A)$ preserves the length function, we obtain
\eq{ \label{wX:diagaut}
\tau \circ w_X = w_X \circ \tau \qq \Longleftrightarrow \qq \tau \in \Aut_X(A)
}
where $\Aut_X(A)$ is the subgroup $\{ \tau \in \Aut(A) \, : \, \tau(X) \subseteq X \}$.
We denote $n_X := n(w_X)$ and obtain that $\tau \in \Aut(A)$ commutes with $\Ad(n_X)$ in $\Aut(\mfg)$ if and only if $\tau \in \Aut_X(A)$.
By \cite[Cor.\ 4.10.3]{BBBR95} we have
\eq{
\label{AdnX:square} \Ad(n_X^2) = \Ad(\ze_X) \qq \text{where} \qq \ze_X := \ze(w_X) \in \wt H.
}
It follows that $\ze_X(\la) = (-1)^{\la(2\rho^\vee_X)}$ for all $\la \in Q$ with $\rho^\vee_X$ half the sum of positive coroots of $\mfg_X$ (the sum of the fundamental coweights of $\mfg_X$).

%%% 2.7.

\subsection{Automorphisms of $\mfg$} \label{sec:Autg}

For any subset $\mft \subseteq \mfg$ we denote by $\Aut(\mfg,\mft)$ the subgroup of $\Aut(\mfg)$ of automorphisms stabilizing $\mft$ and by $\Aut(\mfg;\mft)$ the subgroup of $\Aut(\mfg,\mft)$ of automorphisms fixing $\mft$ pointwise.
Furthermore, if $K$ is a subgroup of $\Aut(\mfg)$ and $\theta,\theta' \in \Aut(\mfg)$ are such that $k \circ \theta = \theta' \circ k$ for some $k \in K$ we say that $\theta$ and $\theta'$ are \emph{$K$-conjugate} and write $\theta \sim_K \theta'$.

The subgroup $\Aut(\mfg;\mfg')$ is the group of \emph{transvections} of $\mfg$ described in \cite[4.20]{KW92};
its elements are of the form $\si_f := \id + f \circ \pi''$ where $f$ is an arbitrary $\F$-linear map: $\mfh'' \to \mfc$ and $\pi''$ is the projection from $\mfg$ onto $\mfh''$ with respect to the decomposition $\mfg = \mfg' \oplus \mfh''$.
Such elements commute with $\Ad(\wt H \ltimes G)$ and satisfy $\tau \circ \si_f \circ \tau^{-1} = \si_{\tau \circ f \circ \tau^{-1}}$ for all $\tau \in \Aut(A)$.
The following then defines a subgroup of $\Aut(\mfg)$:
\eq{
\Inn(\mfg):=\Aut(\mfg;\mfg') \times \Ad(\wt H \ltimes G).
}

The \emph{Chevalley involution} is the unique involutive Lie algebra automorphism $\om$ such that $\om(e_i)=-f_i$ and $\om|_\mfh = -\id_\mfh$.
Note that $\om$ centralizes the subgroups $\Aut(A)$, $\Aut(\mfg;\mfg')$ and $\Ad(N)$ of $\Aut(\mfg)$.
For all $X$ of finite type we have
\eq{
\label{AdnX:gX} \Ad(n_X)|_{\mfg_X} = \om \circ \oi_X|_{\mfg_X},
}
see e.g. \cite[Lem.\ 4.9]{BBBR95}; in particular $\om \in \Aut(A) \ltimes \Inn(\mfg)$ if $A$ is of finite type.
We set $\Out(A)=\Aut(A)$ if $A$ is of finite type and $\Out(A)=\{ \id,\om\} \times \Aut(A)$ otherwise.
The following decomposition of $\Aut(\mfg)$ is established in \cite[4.23]{KW92}:
\eq{ \label{Autg:decomposition}
\Aut(\mfg) = \Out(A) \ltimes \Inn(\mfg).
}

If $\theta \in \Aut(\mfg,\mfh)$ then we denote the dual linear map $(\theta|_\mfh)^*$ on $\mfh^*$ simply by $\theta^*$.
From the definition of root space it follows that $\theta(\mfg_\al) = \mfg_{(\theta^*)^{-1}(\al)}$ for all $\al \in \Phi$.
In particular, $\theta^*$ stabilizes $\Phi$ and $Q$.
Furthermore, $\theta$ acts on $\wt H$ via $(\theta \ast \chi)(\la) = \chi((\theta^*)^{-1}(\la))$ for all $\la \in Q$, $\chi \in \wt H$, see \cite[4.23]{KW92}.
We denote by $\wt H^\theta$ the subgroup of $\wt H$ consisting of those $\chi \in \wt H$ such that $\theta \ast \chi = \chi$.
Note that, since the field $\F$ is closed under taking square roots, one has
\eq{ \label{tildeH:conjugacy}
\forall (\theta,\chi) \in \Aut(\mfg,\mfh) \times \wt H \qu \exists \chi' \in \wt H^\theta: \; \Ad(\chi) \circ \theta \sim_{\Ad(\wt H)} \Ad(\chi') \circ \theta.
}

By \cite[1.16 (i)]{KW92}, $\Aut(\mfg,\mfh) \cap \Ad(G) = \Ad(N)$; since elements of $\Out(A)$, $\Aut(\mfg;\mfg')$ and $\Ad(\wt H)$ all stabilize $\mfh$ we obtain
\eq{ \label{Autgh:decomposition}
\Aut(\mfg,\mfh) = \Out(A) \ltimes \Inn(\mfg,\mfh), \qq \Inn(\mfg,\mfh):=\Aut(\mfg;\mfg') \times \Ad(\wt H \ltimes N).
}
Note that elements of $\Aut(\mfg,\mfh)$ send real root spaces to real root spaces.

By \cite[Thm.\ 3]{PK83}, every Borel subalgebra of $\mfg$ is $\Ad(G)$-conjugate to $\mfb^+$ or $\mfb^-$, where $\mfb^\pm := \langle \mfh, \mfn^\pm \rangle$.
In particular, given $\theta \in \Aut(\mfg)$, $\theta(\mfb^+)$ is $\Ad(G)$-conjugate to $\mfb^+$ or $\mfb^-$ and we say that $\theta$ is \emph{of the first} or \emph{second kind}, respectively.
According to \cite[4.6]{KW92}, $\theta$ is of the second kind if and only if $\theta(\mfb^+) \cap \mfb^+$ is finite-dimensional.
The set of automorphisms of the first kind is a subgroup of $\Aut(\mfg)$ and its coset with respect to $\om$ is the set of automorphisms of the second kind.
By \eqref{AdnX:gX}, if $\mfg$ is finite-dimensional then these two subsets coincide; otherwise they define a partition of $\Aut(\mfg)$.

%%% 2.8.

\subsection{Twisted involutions and compatible decorations} \label{sec:cdec}
 
Let $\theta \in \Aut(\mfg,\mfh)$ be of the second kind such that $\theta^2|_\mfh=\id_\mfh$.
By \eqref{Autgh:decomposition} there exist a linear map $f: \mfh'' \to \mfc$, $\phi \in \wt H$, $w \in W$ and $\tau \in \Aut(A)$ such that
\eq{
\theta = \si_f \circ \Ad(\phi \cdot n(w)) \circ \om \circ \tau.
} 
By $\Aut(\mfg;\mfg')$-conjugacy we may assume that $f$, and hence $\si_f$, commutes with $\tau$.
From the condition that the restriction of $\theta$ to $\mfh$ is involutive we obtain
\eq{
f=0, \qq \tau(w)=w^{-1}, \qq \tau^2 = \id_I.
}
In particular, $w \in W$ is a $\tau$-twisted involution.
Invoking the result \cite[Prop.\ 3.3]{Sp85}, we obtain that there exists $v \in W$ and $X \subseteq I$ of finite type such that $\tau|_X = \oi_X$ and
\eq{
w = v \circ w_X \circ \tau(v)^{-1}.
}
Recalling \eqref{nw:inverse} we obtain
\eq{
\theta = \Ad(\phi \cdot n(v) \cdot n_X) \circ \om \circ \tau \circ \Ad(\ze(w) \cdot n(v)^{-1}).
}
Hence there exists $\chi \in \wt H$, $X \subseteq I$ of finite type and involutive $\tau \in \Aut_X(A)$ such that $\tau|_X = \oi_X$ and
\eq{
\theta \sim_{\Ad(N)} \Ad(\chi \cdot n_X) \circ \om \circ \tau
}
Now by \eqref{tildeH:conjugacy} we may assume that $\chi \in \wt H^{\theta(X,\tau)}$ where
\eq{ \label{theta:def}
\theta(X,\tau) := \Ad(n_X) \circ \om \circ \tau.
}
Note that $\theta(X,\tau)$ is a particular case of a \emph{special semisimple automorphism of the second kind} as defined in \cite[4.38]{KW92}.
We observe that the three factors in \eqref{theta:def} pairwise commute so that $\theta(X,\tau)^2 = \Ad(\ze_X)$ by \eqref{AdnX:square}; in particular, the order of $\theta(X,\tau)$ divides 4.
Also, by \eqref{AdnX:gX}, the condition $\tau|_X = \oi_X$ implies that $\theta(X,\tau)|_{\mfg_X} = \id_{\mfg_X}$, i.e. $\theta(X,\tau)$ fixes pointwise all $\theta(X,\tau)$-stable root spaces.

We are led to the following definition, the natural Kac-Moody analogue of \cite[Eqn.\ (2.8)]{RV20}.

\begin{defn} \label{def:cdec}
Let $X \subseteq I$ and $\tau \in \Aut_X(A)$.
We call $(X,\tau)$ a \emph{compatible decoration} if $X$ is of finite type, $\tau^2=\id_I$ and $\tau|_X = \oi_X$. 
Furthermore, we call $(X,\tau,\chi)$ an \emph{enriched compatible decoration} if $(X,\tau)$ is a compatible decoration and $\chi \in \wt H^{\theta(X,\tau)}$. 
If $(X,\tau,\chi)$ is an enriched compatible decoration, in a minor abuse of notation we will write $\theta(X,\tau,\chi) = \Ad(\chi) \circ \theta(X,\tau)$.
\hfill \defnend
\end{defn}

The following result summarizes the above discussion.

\begin{prop} \label{prop:thetafactorization} 
Let $\theta \in \Aut(\mfg,\mfh)$ be a pseudo-involution of the second kind.
Then there exists an enriched compatible decoration $(X,\tau,\chi)$ such that $\theta \sim_{\Inn(\mfg,\mfh)} \theta(X,\tau,\chi)$.
\end{prop}

Compatible decorations are indicated diagrammatically by decorating the underlying Dynkin diagram as follows: fill the nodes corresponding to $X$ and indicate the nontrivial $\tau$-orbits by bidirectional single arrows.
The following basic properties will be useful later.

\begin{lemma} \label{lem:Xtau:basics}
Let $(X,\tau)$ be a compatible decoration, let $i \in I \backslash X$ and let $Z$ be the union of connected components of $X \cup \{ i,\tau(i) \}$ containing $i$ or $\tau(i)$.
The following statements are true.
\begin{enumerate}
\item \label{lem:Xtau:basics:connected}
The pair $(X \cap Z,\tau|_Z)$ is a compatible decoration of the generalized Cartan matrix $A_Z$, and either $Z$ is of type ${\sf A}_1 \times {\sf A}_1$ or $Z$ is connected.
\item \label{lem:Xtau:basics:theta:identity}
One has the following identity for linear maps on $\mfh$ (or linear maps on $\mfh^*$):
\eq{ \label{theta:identity}
(\theta(X,\tau)-\id) \circ (\tau-\id) = 0.
}
\item \label{lem:Xtau:basics:wX:formula}
If $\{ \ka^\vee_j \}_{j \in X}$ denotes the set of fundamental coweights of $\mfg_X$, then
\eq{ \label{wX:formula}
w_X(\al_i) = \al_i + \sum_{j \in X} v_{ij} \al_j, \qq v_{ij} = -(\al_i + \al_{\tau(i)})(\ka^\vee_j) \in \Z_{\ge 0}.
}
\end{enumerate}
\end{lemma}

\begin{proof} \mbox{}
\begin{enumerate}[leftmargin=!,font={\itshape},labelwidth=\widthof{\emph{(iii)}}]
\item[{\ref{lem:Xtau:basics:connected}}]
Suppose $i \notin X^\perp$ and let $j \in X$ such that $a_{ij} \ne 0$.
Let $X'$ be the connected component of $X$ containing $j$, so that $i \notin (X')^\perp$.
Then $\tau|_{X'} = \oi_X|_{X'} = \oi_{X'}$.
Taking the union of all such $X'$, we obtain the first statement.
Furthermore, from $a_{\tau(i)\tau(j)} = a_{ij}$ it follows that $\tau(i) \notin (X')^\perp$ so that $X' \cup \{ i,\tau(i)\}$ is connected.
Now suppose $i \in X^\perp$.
By the above analysis, also $\tau(i) \in X^\perp$.
Consider the full Dynkin subdiagram whose set of vertices is $Z = \{ i,\tau(i)\}$.
If $Z$ is not connected, it must be of type ${\sf A}_1 \times {\sf A}_1$.
\item[{\ref{lem:Xtau:basics:theta:identity}}] 
The linear map $(\theta(X,\tau)-\id) \circ (\tau-\id) = (w_X-\id) \circ (\tau-\id)$ maps into $\mfh_X$; on the other hand its image lies in $\mfh^{-\theta(X,\tau)}$. 
Since $\theta(X,\tau)|_{\mfh_X} = \id_{\mfh_X}$ it follows that $(\theta(X,\tau)-\id) \circ (\tau-\id) =0$.
A similar argument applies to the action on $\mfh^*$.
\item[{\ref{lem:Xtau:basics:wX:formula}}]
It is clear that the first equality in \eqref{wX:formula} holds for some nonnegative integers $v_{ij}$. 
Now apply that equality to the fundamental coweight $\ka^\vee_j$ for $j \in X$.
It yields $v_{ij} = (w_X(\al_i))(\ka^\vee_j) - \al_i(\ka^\vee_j)$.
Since $(X,\tau)$ is a compatible decoration, we readily obtain $(w_X(\al_i))(\ka^\vee_j) = \al_i(w_X(\ka^\vee_j)) = -\al_i(\tau(\ka^\vee_j)) = -\al_{\tau(i)}(\ka^\vee_j)$ and we arrive at the desired expression for $v_{ij}$ in terms of $\ka^\vee_j$. \qedhere
\end{enumerate}
\end{proof}

Satake diagrams (of Kac-Moody type) are particular types of compatible decorations, associated to almost split real forms of Kac-Moody algebras or equivalently to involutive automorphisms of $\mfg$ of the second kind (see \cite[Def. 4.10 (b), Cor. 4.10.4]{BBBR95} and \cite[Def.\ 2.3]{Ko14}).
We can restate this definition as follows.

\begin{defn} \label{def:Sat}
We call $i \in I \backslash X$ an \emph{odd node} if $\tau(i)=i$ and $\ze_X(\al_i)=-1$.
A compatible decoration is called a \emph{Satake diagram} if there are no odd nodes.\hfill \defnend
\end{defn}

%%% 2.9.

\subsection{Classification of pseudo-involutions of the second kind} \label{sec:pseudoinv2nd:class}

Now let $\theta \in \Aut(\mfg)$ be an arbitrary pseudo-involution of the second kind and let $\mft$ be a $\theta$-stable Cartan subalgebra of $\mfg$ such that $\theta|_\mft$ is an involution.
By \cite[Thm.\ 2(a)]{PK83}, $\mfh$ is $\Ad(G)$-conjugate to $\mft$.
We can now apply Proposition \ref{prop:thetafactorization} and hence deduce the following main result of this section.

\begin{thrm} \label{thm:pseudoinv:class}
Any pseudo-involution of the second kind is $\Inn(\mfg)$-conjugate to $\theta(X,\tau,\chi)$ for some enriched compatible decoration $(X,\tau,\chi)$.
\end{thrm}
Theorem \ref{thm:pseudoinv:class} is a special case of \cite[4.39]{KW92} with stronger constraints on the datum $(X,\tau)$.
It can be restated as follows: the assignment $(X,\tau,\chi) \mapsto \theta(X,\tau,\chi)$ induces a surjection:
\eq{ 
\{ \text{enriched compatible decorations} \} \to \{ \text{pseudo-involutions of the 2nd kind} \} / \Inn(\mfg).
}
We can compare Theorem \ref{thm:pseudoinv:class} with \cite[Prop.\ A.6]{Ko14}, the analogous statement about involutions of the second kind in terms of Satake diagrams $(X,\tau)$.
In that case without loss of generality we may let $\chi$ take values among fourth roots of unity, see \cite[(2.7)]{Ko14}, in order to guarantee the involutiveness of $\theta(X,\tau,\chi)$. 

Note that $\Aut(A)$ acts on enriched compatible decorations via
\eq{ \label{AutA:actionCDec}
\psi \cdot (X,\tau,\chi) = (\psi(X),\psi \circ \tau \circ \psi^{-1},\psi \ast \chi), \qq \psi \in \Aut(A).
}
Also, $\om$ commutes with $\theta(X,\tau)$ for all compatible decorations $(X,\tau)$ and $\om \ast \chi = \chi^{-1} \in \wt H^{\theta(X,\tau)}$ for all $\chi \in \wt H^{\theta(X,\tau)}$.
Finally, note that all elements of $\Out(A)$ act by conjugation on the set of pseudo-involutions of the second kind.
Hence we obtain from Theorem \ref{thm:pseudoinv:class} the following result.

\begin{crl} \label{cor:pseudoinv:diagaut:class}
The assignment $(X,\tau,\chi) \mapsto \theta(X,\tau,\chi)$ induces a surjection:
\[
\ol{\theta}: \{ \text{enriched compatible decorations} \} / \Aut(A) \to \{ \text{pseudo-involutions of the 2nd kind} \} / \Aut(\mfg).
\]
\end{crl}

In the same way, \cite[Prop.\ A.6]{Ko14} can be promoted to the existence of a surjection from the set of $\Aut(A)$-orbits of Satake diagrams to the set of $\Aut(\mfg)$-conjugacy classes of involutions of the second kind.
Owing to results in \cite[Ch.\ 5]{KW92} about involutive automorphisms of $\mfg$, this surjection is in fact a bijection for any $A$, see \cite[Thm.\ 2.7]{Ko14}.

However, the map $\ol{\theta}$ in Corollary \ref{cor:pseudoinv:diagaut:class} is not always injective.
Indeed, if $A$ is such that there exists a compatible diagram $(X,\tau)$ with at least one odd node, then there always exists a choice of $\chi$ such that $\theta(X,\tau,\chi)$ is an involution. 
For instance, fix a square root $\sqrt{-1} \in \F$ of $-1$ and let $\chi \in \wt H$ be defined by $\chi(\al_i) = (\sqrt{-1})^{2\al_i(\rho^\vee_X)}$ for all $i \in I$.
It can be checked using the formula \eqref{wX:formula} and the relation $\rho^\vee_X = \sum_{j \in X} \ka^\vee_j$ that $\chi \in \wt H^{\theta(X,\tau)}$.
Moreover, we have $\theta(X,\tau,\chi)^2 = \Ad(\chi^2 \cdot \ze_X) = \id_\mfg$, as claimed.
By \cite[Thm.\ 2.7]{Ko14}, $\theta(X,\tau,\chi)$ is $\Aut(\mfg)$-conjugate with some $\theta(X',\tau',\chi')$ where $(X',\tau')$ is a Satake diagram and $\chi'$ is the character $s(X,\tau)$ defined in \cite[(2.7)]{Ko14}.
Since the $\Aut(A)$-action on the set of compatible decorations preserves the set of Satake diagrams, injectivity does not hold.

Conversely, if $\chi \in \wt H$ is such that $\theta(X,\tau,\chi)^2=\id_\mfg$ and $i$ is an odd node then $\chi(\al_i-w_X(\al_i))=-1$, so $\chi(\al_j) \ne 1$ for some $j \in X$.
Hence, any involution of the form $(X,\tau,\chi)$ where $(X,\tau)$ is not a Satake diagram does not fix pointwise $\mfg_X$.
It is therefore natural to focus on pseudo-involutions of the second kind which fix pointwise all stable root spaces.
Straightforwardly one obtains the following refinement of Corollary \ref{cor:pseudoinv:diagaut:class}.

\begin{crl} \label{crl:pseudoinv:class:2}
Any pseudo-involution of the second kind which fixes pointwise all stable root spaces is $\Inn(\mfg)$-conjugate to $\theta(X,\tau,\chi)$ for some enriched compatible decoration $(X,\tau,\chi)$ with $\chi|_{Q_X}=1$.
\end{crl}

The set of enriched compatible decorations $(X,\tau,\chi)$ such that $\chi|_{Q_X}=1$ is $\Aut(A)$-stable.
We conjecture that only allowing maps which fix pointwise all stable root spaces resolves the failure of the surjection $\ol{\theta}$ in Corollary \ref{cor:pseudoinv:diagaut:class} to be a bijection.

\begin{conj} \label{conj:pseudoinv:class:2}
The assignment $(X,\tau,\chi) \mapsto \theta(X,\tau,\chi)$ induces a bijection $\ol{\theta}_\mathsf{fix}$ from 
\[
\{ \text{enriched compatible decorations } (X,\tau,\chi) \text{ such that } \chi|_{Q_X}=1 \} / \Aut(A) \]
to
\[
 \{ \text{pseudo-involutions of the 2nd kind fixing pointwise all stable root spaces} \} / \Aut(\mfg). 
\]
\end{conj}

Note that if $\psi \in \Aut(A)$ and $(X,\tau,\chi)$ is a compatible decoration such that $\chi|_{Q_X}=1$ then $\theta(\psi \cdot (X,\tau,\chi)) = \psi \circ \theta(X,\tau,\chi) \circ \psi^{-1}$.
Considering Corollary \ref{crl:pseudoinv:class:2} we see that in order to prove Conjecture \ref{conj:pseudoinv:class:2}, it suffices to establish the injectivity of $\ol{\theta}_\mathsf{fix}$.
Note that pseudo-involutions $\theta$ of the second kind which fix pointwise stable root spaces automatically have a split pair, see \cite[5.15]{KW92}, so that a proof along the lines of the last part of \cite[App.\ A]{Ko14} may be possible.
On the other hand, the crucial result \cite[Thm.\ 5.31]{KW92} does not directly extend to pseudo-involutions: it is no longer true that $\mfg = \mfb^+ + \mfg^\theta$ if $\theta = \theta(X,\tau,\chi)$ is not an involution.
Nevertheless, it is useful to follow this approach in first instance. 

Consider a pair of enriched compatible decorations $(X,\tau,\chi)$ and $(X',\tau',\chi')$ such that $\chi|_{Q_X} = \chi'|_{Q_{X'}} = 1$ and
\eq{ \label{Autg:conjugacy:pseudo}
\theta \circ \phi = \phi \circ \theta'
}
for some $\phi \in \Aut(\mfg)$, where $\theta:=\theta(X,\tau,\chi)$ and $\theta':=\theta(X',\tau',\chi')$.
We wish to show that there exists $\psi \in \Aut(A)$ such that 
\eq{ \label{cdec:conjugacy:pseudo}
\psi \cdot (X',\tau',\chi') = (X,\tau,\chi).
}
Using the decomposition \eqref{Autg:decomposition} and observing that $\theta$ preserves $\mfg'$ and satisfies $\theta \circ \si_f = \si_{\tau \circ f \circ \tau^{-1}} \circ \theta$ for all linear $f: \mfh'' \to \mfc$, we may assume that $\phi \in \Aut(\mfg') = \Out(A) \ltimes \Ad(\wt H \ltimes G)$. 
Furthermore, by replacing $\phi$ by $\theta^{-1} \circ \phi$ if necessary, we may assume that $\phi$ is of the first kind, i.e.\ $\phi$ lies in 
\eq{
\Aut_{\rm I}(\mfg') := \Aut(A) \ltimes \Ad(\wt H \ltimes G).
}
Denoting the factor of $\phi$ in $\Aut(A)$ by $\psi$, from \eqref{Autg:decomposition} we deduce 
\eq{ \label{Autg:conjugacy:pseudo:diagaut}
\tau \circ \psi = \psi \circ \tau'
} 
as desired.
It would now be sufficient to prove that $\psi(X')=X$, $\psi \ast \chi' = \chi$ and $\phi \circ \psi^{-1} \in \Ad(\wt H \ltimes G)$ commutes with $\theta$.\\

We make one more general remark before proving the conjecture in a special case.
Consider an enriched compatible decoration $(X,\tau,\chi)$ such that $\chi|_{Q_X}=1$.
Fix $i \in I \backslash X$.
The conditions on $\chi$ imply that $\chi(\al_i+\al_{\tau(i)})=1$.
Suppose in addition (if $\tau(i)=i$ this is automatic) that $\chi(\al_i)=\chi(\al_{\tau(i)})$, necessarily equal to $\pm 1$.
Then there exists $\chi_+ \in \wt H^{\theta(X,\tau)}$ such that 
\eq{ \label{Autg:conjugacy:pseudo:signs}
\theta(X,\tau,\chi) \sim_{\Ad(\wt H)} \theta(X,\tau,\chi_+), \qq \chi_+(\al_i)=\chi_+(\al_{\tau(i)})=1.
}
Indeed if $\chi(\al_i) = \chi(\al_{\tau(i)}) = -1$, one can conjugate by $\Ad(\xi)$ where $\xi \in \wt H$ satisfies $\xi(\al_i)=\xi(\al_{\tau(i)}) = \sqrt{-1}$ and $\xi(\al_j)=1$ if $j \ne \{ i,\tau(i) \}$.\\

To provide evidence for the conjecture, we now give an elementary proof in the crucial case $\mfg = \mfsl_3$ (the smallest indecomposable Kac-Moody algebra for which there exist non-involutive pseudo-involutions).
Similar proofs, all based on the fact that conjugation of an automorphism preserves its order and the dimension of its eigenspaces, can straightforwardly be given for other low-rank cases.

Let $I=\{1,2\}$ and $A=(a_{ij})_{i,j \in I}$ be such that $a_{11}=a_{22}=2$ and $a_{12}=a_{21}=-1$.
In addition to the Satake diagrams $(\emptyset,\id)$, $(\emptyset,\oi_I)$, $(I,\oi_I)$, where $\oi_I$ is the nontrivial diagram automorphism, we have the compatible decorations $(\{1\},\id)$ and $\oi_I \cdot (\{1\},\id) = (\{2\},\id)$. 
Consider two enriched compatible decorations $(X,\tau,\chi)$ and $(X',\tau',\chi')$ such that $\chi|_{Q_X} = \chi'|_{Q_{X'}}=1$.
Suppose that $\theta(X,\tau,\chi) \sim_{\Aut_{\rm I}(\mfg')} \theta(X',\tau',\chi')$.
We must show that $(X,\tau,\chi)$ and $(X',\tau',\chi')$ lie in the same $\Aut(A)$-orbit.
By \eqref{Autg:conjugacy:pseudo:diagaut}, we only need to show $(X,\chi)=(X',\chi')$ or $(X,\chi) = \oi_I \cdot (X',\chi')$. 

\begin{description}
\item[$\tau=\tau'=\id$]
By \eqref{Autg:conjugacy:pseudo:signs} we may assume $\chi=\chi'=1$.
Hence it suffices to study the conjugacy problem 
\eq{
\theta(X,\id) \sim_{\Aut_{\rm I}(\mfg')} \theta(X',\id).
}
Recall that $\theta(X,\tau)^2 = \Ad(\ze_X)$ and note that $X$ is either empty or a singleton.
Hence $\theta(X,\id)$ is an involution if and only if $X=\emptyset$.
As a consequence, either $X=\emptyset=X'$, and we are done, or $X \ne \emptyset \ne X'$.
In the second case both $X$ and $X'$ are singletons and either $X=X'$, in which case $(X,\tau,\chi) = (X',\tau',\chi')$, or $X = \oi_I(X')$, in which case $(X,\tau,\chi) = \oi_I \cdot (X,\tau',\chi')$, as required.

\item[$\tau=\tau'=\oi_I$]
Note that in this case each of $X$ and $X'$ are equal to $I$ or $\emptyset$, and it is necessary to show that $X=X'$. 
On the one hand, $\theta(I,\oi_I) = \id$ so that $\theta(I,\oi_I,\chi)=\id$ for any $\chi \in \wt H^{\theta(I,\oi_I)}$ such that $\chi|_{Q_I} = 1$ (i.e.\ for $\chi=1$).
On the other hand, $\theta(\emptyset,\oi_I,\chi)$ sends $h_1$ to $-h_2$, and hence is not the identity, for any $\chi \in \wt H$.
Therefore $\theta(X,\oi_I,\chi) \sim_{\Aut_{\rm I}(\mfg')} \theta(X',\oi_I,\chi')$ implies $X=X'$.
If $X=X'=I$, we automatically have $\chi=\chi'(=1)$, as required.\\

To deal with the case $X=X'=\emptyset$, note that $\theta(\emptyset,\oi_I,\chi)$ is an involution if and only if $\chi^2=1$, which is equivalent to $\chi(\al_1)=\chi(\al_2) \in \{ \pm 1 \}$.
Again by \eqref{Autg:conjugacy:pseudo:signs}, we may assume that both $\chi$ and $\chi'$ are equal to 1; we obtain $(X,\tau,\chi) = (X',\tau',\chi')$ as required.

It remains to deal with the case that $\chi^2 \ne 1 \ne (\chi')^2$.
It suffices to prove that $\chi(\al_1)=\chi'(\al_1)$.
We observe that $\Ad(x)$ sends $\mfg^{\theta'}$ to $\mfg^\theta$, as a consequence of \eqref{Autg:conjugacy:pseudo}. 
In this case $\mfg^\theta = \mfg^{\theta'}$ is the abelian subalgebra $\F (h_1-h_2) \oplus \F b_{12}$ where $b_{12}:=[f_1,f_2]-[e_1,e_2]$.
Now we identify $\mfsl_3$ with its faithful 3-dimensional representation according to
\eq{
e_1 = \begin{pmatrix} 0 & 1 & 0 \\ 0 & 0 & 0 \\ 0 & 0 & 0 \end{pmatrix}, \qu f_1 = \begin{pmatrix} 0 & 0 & 0 \\ 1 & 0 & 0 \\ 0 & 0 & 0 \end{pmatrix}, \qu e_2 = \begin{pmatrix} 0 & 0 & 0 \\ 0 & 0 & 1 \\ 0 & 0 & 0 \end{pmatrix}, \qu f_2 = \begin{pmatrix} 0 & 0 & 0 \\ 0 & 0 & 0 \\ 0 & 1 & 0 \end{pmatrix}.
}
Hence $x$ is identified with an element of $\GL(3,\F)$ (well-defined up to a scalar multiple).
The condition on $x$ that $\Ad(x)$ preserves $\mfg^\theta$ has two solution classes.
\begin{enumerate}
\item 
In the first class of solutions we have 
\eq{
\begin{gathered}
\Ad(x)(h_1-h_2)=h_1-h_2, \qq \Ad(x)(b_{12})=\eps b_{12}, \\
x=\begin{pmatrix} A & 0 & B \\ 0 & 1 & 0 \\ \eps B & 0 & \eps A \end{pmatrix},
\end{gathered}
}
where $\eps \in \{ \pm 1 \}$ and $A,B \in \F$ are such that $A^2 \ne B^2$.
Now we obtain 
\eq{
\Ad(x)(f_1) = (A^2-B^2)^{-1}(A f_1 - \eps B e_2), \qq \Ad(x)(e_2) = (A^2-B^2)^{-1}(- B f_1 + \eps A e_2).
}
In the case $\psi=\id$, applying \eqref{Autg:conjugacy:pseudo} to $f_1$ gives $\chi(\al_2)=\chi'(\al_2)$, as required.
If $\psi=\oi_I$, applying \eqref{Autg:conjugacy:pseudo} to $f_2$ yields a contradiction, so \eqref{Autg:conjugacy:pseudo} does not hold in this case and there is nothing to prove.
\item
In the second class, we find 
\eq{ 
\begin{gathered}
\Ad(x)(h_1-h_2) = -\frac{1}{2} \big( h_1-h_2+3 \ze b_{12} \big), \qq 
\Ad(x)(b_{12})  = -\frac{\eta}{2} \big( h_1-h_2 - \ze b_{12} \big), \\
x=\begin{pmatrix} 1 & C & \eta \\ -\eta D & 0 & D \\ \ze & -\ze C & \ze \eta \end{pmatrix}
\end{gathered}
}
where $\ze, \eta \in \{\pm 1 \}$ and $C,D \in \F^\times$.
Now we obtain 
\eq{ 
\begin{aligned}
\Ad(x)(f_1) &= \frac{C}{4} \Big( h_1+h_2 + \ze \big( [e_1,e_2]+[f_1,f_2] \big) + 2\eta D^{-1} \big( \ze f_2 - e_1 \big) \Big),
\\
\Ad(x)(e_2) &= \frac{\eta C}{4} \Big( h_1+h_2 + \ze  \big( [e_1,e_2]+[f_1,f_2] \big) - 2 \eta D^{-1} \big( \ze f_2 - e_1 \big) \Big).
\end{aligned}
}
Applying \eqref{Autg:conjugacy:pseudo} to $f_{\psi(1)}$ and considering the projections to $\mfg_{\al_1}$ and $\mfg_{-\al_2}$ results in a contradiction.
Hence \eqref{Autg:conjugacy:pseudo} does not hold and there is nothing to prove.
\end{enumerate}
\end{description}

%%%%%%%%%%%%%%%%%%%%%%%%%%%%%%%%%%%%%%%%%%%%%%%%%%%%%%%%%%%%%%%%%%
% EoF
%%%%%%%%%%%%%%%%%%%%%%%%%%%%%%%%%%%%%%%%%%%%%%%%%%%%%%%%%%%%%%%%%%

%%%%%%%%%%%%%%%%%%%%%%%%%%%%%%%%%%%%%%%%%%%%%%%%%%%%%%%%%%%%%%%%%%
% Section 3
%%%%%%%%%%%%%%%%%%%%%%%%%%%%%%%%%%%%%%%%%%%%%%%%%%%%%%%%%%%%%%%%%%

\section{Pseudo-fixed-point subalgebras in terms of generalized Satake diagrams} \label{sec:pseudofixedpoint}

In this section we discuss constructions of pseudo-fixed-point subalgebras in the Kac-Moody setting, associated to pseudo-involutions of the second kind.
In \cite[Sec.~3]{RV20} we considered a Lie subalgebra $\mfk = \mfk(X,\tau,\chi) \subseteq \mfg$ of finite type which is a pseudo-fixed-point subalgebra if and only if the enriched compatible decoration $(X,\tau,\chi)$ satisfies certain additional conditions. 
Part of this earlier work immediately generalizes to the Kac-Moody setting, see \cite[Rmk.~3 (ii)]{RV20}. 
Here we give a brief synopsis of it for completeness.
This is complemented with a statement of the Iwasawa decomposition for a pseudo-symmetric pair.

%%% 3.1.

\subsection{The subalgebra $\mfk$}

Let $\mft$ be an arbitrary Cartan subalgebra of the Kac-Moody algebra $\mfg$. 
Denote by $\mfg_\al^{(\mft)}$ the $\mft$-root space associated to $\al \in \mft^*$,
\eq{
\mfg_\al^{(\mft)} := \{ x \in \mfg \, : \,\forall t \in \mft \; [t,x] = \al(t) x \}. \label{mft-root-space}
} 
and by $\Phi^{(\mft)}$ the corresponding root system,
\eq{
\Phi^{(\mft)} := \{ \al \in \mft^* \, : \, \mfg_\al^{(\mft)} \ne \{0\} \}  \backslash \{0\}. \label{mft-root-system}
}
Choose a subset $\Pi^{(\mft)} \subset \Phi^{(\mft)}$ such that $\Phi^{(\mft)} = \Sp_{\Z_{\ge 0}} \Pi^{(\mft)} \cup \Sp_{\Z_{\le 0}} \Pi^{(\mft)}$.
Moreover, for every $\al \in \Pi^{(\mft)}$ and every choice of sign, choose nonzero elements $x_{\pm \al} \in \mfg^{(\mft)}_{\pm \al}$ (simple root vectors).
Then $\mfg$ is generated by $\mft$ and $\{x_\al,x_{-\al}\}_{\al \in \Pi^{(\mft)}}$.

\begin{defn} \label{def:pseudofixedpoint:KM}
With the notation defined above, let $\theta \in \Aut(\mfg)$ be a pseudo-involution of the second kind stabilizing $\mft$ such that $\theta|_\mft$ is an involution. 
To this datum we associate a subalgebra $\mfk$, defined to be the subalgebra generated by $\mft^\theta$, all $x_\al$ and $x_{-\al}$ such that $\al \in \Pi^{\theta^*}$, and the elements $x_\al+\theta(x_\al)$ for $\al \in \Pi \backslash \Pi^{\theta^*}$.
\hfill \defnend
\end{defn}

This generating set is natural in view of Definition \ref{def:pseudofixedpoint} and shortly we will prove that under an additional mild assumption on $(X,\tau,\chi)$ the subalgebra $\mfk$ is indeed a pseudo-fixed-point subalgebra.

\begin{rmk}
Definition \ref{def:pseudofixedpoint:KM} also makes sense if $\theta$ is of the first kind, but then $\mfk$ will typically not be a pseudo-fixed-point subalgebra in the sense of Definition \ref{def:pseudofixedpoint}.
This can be easily checked if $\mfg$ is of type $\wh{\sf A}_1$ and $\theta$ is the nontrivial diagram automorphism, cf. \cite[Example 5.11]{KW92}. \hfill \rmkend 
\end{rmk}

If $\theta$ is a pseudo-involution of the second kind then Theorem \ref{thm:pseudoinv:class} implies that by $\Aut(\mfg)$-conjugacy we may assume 
\eq{
\theta = \Ad(\chi) \circ \theta(X,\tau)
}
for some enriched compatible decoration $(X,\tau,\chi)$.
Then the corresponding subalgebra $\mfk$ is generated by $\mfh^\theta$, $\mfn^+_X$ and the elements 
\eq{ \label{bi:def}
b_i := \begin{cases} f_i & \text{if } i \in X, \\ f_i + \theta(f_i) & \text{otherwise},
\end{cases}
}
which recovers the description of a fixed-point subalgebra $\mfg^\theta$ in the case that $\theta$ is involutive, see \cite[Lem.~2.8]{Ko14}. 

\begin{rmk}
Equally, we could have interchanged the role of $\mfn^+$ and $\mfn^-$ in the definition of $\mfk$.
The current choice is customary in recent work on q-deformations of $U\mfk$, see e.g. \cite{Ko14,BK19,DK19,RV20,AV20}. 
\hfill \rmkend
\end{rmk}

%%% 3.2.

\subsection{Generalized Satake diagrams} \label{sec:GSats}

We now consider a particular type of (enriched) compatible decoration; we will see that, if and only if it is associated to such datum, the subalgebra $\mfk = \langle \mfh^\theta, \mfn^+_X, \{ b_i \}_{i \in I} \rangle$ is a pseudo-fixed-point subalgebra.
For all $i \in I \backslash X$ we consider the set 
\eq{
X[i] := X \cup \{ i, \tau(i) \},
}
which we may indentify with the corresponding full subdiagram of $I$.

\begin{defn} \label{def:GSat} \mbox{}
\begin{enumerate}
\item \label{def:GSat:vanilla}
Let $(X,\tau)$ be a compatible decoration.
We call $(X,\tau)$ a \emph{generalized Satake diagram} if for all $i \in I \backslash X$ such that $\tau(i)=i$ the connected component of the subdiagram $X[i]$ containing $i$ is not of type ${\sf A}_2$ (i.e.~$X[i]$ does not have
\tp[baseline=-3pt,line width=.7pt,scale=75/100]{
\draw (.1,0) --  (.5,0);
\filldraw[fill=white] (0,0) circle (.1);
\filldraw[fill=black] (.5,0) circle (.1);
} 
as a connected component).
\item \label{def:GSat:enriched}
Let $(X,\tau,\chi)$ be an enriched compatible decoration.
If $(X,\tau)$ is a generalized Satake diagram and $\chi(\al_{\tau(i)}) = \chi(\al_i)$ for all $i \in X^\perp$ such that $a_{i \tau(i)}=0$ then we call $(X,\tau,\chi)$ an \emph{enriched generalized Satake diagram}. \hfill \defnend
\end{enumerate}
\end{defn}

\begin{rmk} \label{rmk:GSats} \mbox{}
\begin{enumerate}
\item \label{rmk:GSats:Satake}
Recall the notion of a Satake diagram given in Definition \ref{def:Sat}.
If a compatible decoration $(X,\tau)$ is not a generalized Satake diagram, then there exists $(i,j) \in (I \backslash X) \times X$ such that $i,j \in (X\backslash\{j\})^\perp$, $\tau(i)=i$ and $a_{ij}=a_{ji}=-1$.
Hence $\zeta_X(\al_i) = (-1)^{\al_i(h_j)} = -1$, so that $i$ is an odd node.
It follows that all Satake diagrams are generalized Satake diagrams.
\item \label{rmk:GSats:AutA}
The action of $\Aut(A)$ on the set of (enriched) compatible decoration given in \eqref{AutA:actionCDec} stabilizes the set of (enriched) generalized Satake diagrams.
\hfill \rmkend
\end{enumerate}
\end{rmk}

We give the classification of generalized Satake diagrams when $A$ is of finite or affine type in Appendices \ref{sec:tables:finite} and \ref{sec:tables:affine}.
Here we make some remarks to support this classification in the case that $A$ is of classical Lie type. 
Depending on the Lie type, the corresponding generalized Satake diagrams naturally organize themselves into 2, 3 or 4 families. 

To describe these families, it is convenient to say that a connected component of $X$ is \emph{$(X,\tau)$-simple} if it is a singleton $\{j\}$ and there exists $i \in I \backslash X$ such that $a_{ij}a_{ji} = 1$ and $\tau(i)=i$.
Suppose that $\tau$ is trivial except possibly on subsets of type ${\sf A}_1 \times {\sf A}_1$ (these occur when $A$ is of type ${\sf D}_n$, $\wh{\sf B}_n$, $\wh{\sf B}_n^\vee$ or $\wh{\sf D}_n$).
Then the condition that $\tau|_X = \oi_X$ requires that $X$ has at most two connected components which are not $(X,\tau)$-simple, which must be of type ${\sf B}_p$, ${\sf C}_p$ or ${\sf D}_p$ (with $p \le n$).
More precisely, if $A$ is of type ${\sf A}_n$ or $\wh{\sf A}_n$, then all connected components of $X$ are $(X,\tau)$-simple; if $A$ is of type ${\sf B}_n$, ${\sf C}_n$ or ${\sf D}_n$ then $X$ can have at most one connected component which is not $(X,\tau)$-simple; in all other cases $X$ can have two connected components which are not $(X,\tau)$-simple.
The defining condition of generalized Satake diagrams now require that either there are no $(X,\tau)$-simple components, in which case we call $(X,\tau)$ \emph{plain}, or, due to the definining condition of generalized Satake diagram, there are as many as fit in the complement of the other connected components of $X$, in which case we call $(X,\tau)$ \emph{alternating}.

If $A$ is of type $\wh{\sf A}_n$, $\wh{\sf C}_n$, $\wh{\sf C}_n^\vee$ or $\wh{\sf D}_n$ then rotation by a half-turn is a diagram automorphism $\tau$.
To extend $\tau$ to a generalized Satake diagram $(X,\tau)$, in the first case $X$ must be empty and otherwise of type ${\sf A}_p$.
In this case we call $(X,\tau)$ \emph{rotational}.
Finally, if $A$ is of type ${\sf A}_n$ or $\wh{\sf A}_n$ then $\tau$ can be a reflection of the Dynkin diagram, in which case $X$ is of type ${\sf A}_p$ or ${\sf A}_{p_1} \times {\sf A}_{p_2}$, respectively; we call such generalized Satake diagrams \emph{reflecting}.

%%% 3.3.

\subsection{Basic properties of $\mfk$}

In this section we analyse the subalgebra $\mfk$.
In order to do that, for $\bm i = (i_1,\ldots,i_\ell) \in I^\ell$ ($\ell \in \Z_{> 0}$)~we~set
\begin{align}
e_{\bm i} &= [e_{i_1},[e_{i_2},\ldots,[e_{i_{\ell-1}},e_{i_\ell}]\cdots]] \in \mfn^+, \\
f_{\bm i} &= [f_{i_1},[f_{i_2},\ldots,[f_{i_{\ell-1}},f_{i_\ell}]\cdots]] \in \mfn^-, \\
b_{\bm i} &= [b_{i_1},[b_{i_2},\ldots,[b_{i_{\ell-1}},b_{i_\ell}]\cdots]], \\
\al_{\bm i} &= \al_{i_1} + \ldots + \al_{i_\ell} \in Q^+.
\end{align}
Fix a subset $\mc J \subset \cup_{\ell \in \Z_{>0}} I^\ell$ such that $\{ f_{\bm i} \, : \, \bm i \in \mc J \}$ is a basis for $\mfn^-$.
In particular $I \subseteq \mc J$.

Recall the standard order on $Q$, defined for any $\la,\mu \in Q$ by $\la \ge \mu$ if and only if $\la-\mu \in Q^+$.
Finally, for $i,j \in I$ such that $i \ne j$ we set $M_{ij} := 1-a_{ij} \in \Z_{> 0}$ and denote 
\eq{
\la_{ij} := \al_i + M_{ij} \al_j \in Q^+,
}
which is not a root (this statement is equivalent to the Serre relations in $\mfg$).

\begin{lemma} \label{lem:b:triangular}
Let $(X,\tau,\chi)$ be an enriched compatible decoration and let $\theta = \Ad(\chi) \circ \theta(X,\tau)$.
Then 
\eq{
b_{\bm i} - f_{\bm i} \in \bigoplus_{\al \in \Phi \atop \al > -\al_{\bm i}} \mfg_\al
}
for all $\bm i \in I^\ell$, $\ell \in \Z_{>0}$.
Hence, the projection of $b_{\bm i}$ on $\mfg_{-\al_{\bm i}}$ with respect to the root space decomposition \eqref{g:rootspacedecomposition} is $f_{\bm i}$.
\end{lemma}

\begin{proof}
A straightforward induction with respect to $\ell$ and the explicit formula \eqref{bi:def}.
\end{proof}

The generators $b_i$ satisfy Serre relations which, unlike those for $f_i$, may have additional lower order terms.

\begin{lemma}{\cite[Eq.\ (3.7)]{RV20}} \label{lem:bb:formulas}
Let $(X,\tau,\chi)$ be an enriched compatible decoration.
Let $i,j \in I$ be such that $i \ne j$.
\begin{enumerate}
\item \label{lem:bb:formulas:case1}
Suppose $\theta^*(\al_i)+\al_i+\al_j \in -\Phi^+ \cup \{0\}$.
Then $i \in I \backslash X$, $\tau(i)=i$, $j \in X$, $\theta^*(\al_i)+\al_i+\al_j \in -\Phi^+_X \cup \{0\}$ and
\eq{
\ad(b_i)^{M_{ij}}(b_j) = \begin{cases}
(1+\ze_X(\al_i)) [\theta(f_i),[f_i,f_j]] \in \mfn^+_X & \text{if } \theta^*(\al_i)+\al_i+\al_j <0, \, a_{ij}=-1, \\
-18\chi(\al_i)^{-2}e_j & \text{if } \theta^*(\al_i)+\al_i+\al_j=0, \, a_{ij}=-3, \\
-\chi(\al_i)^{-1}(2h_i+h_j) & \text{if } \theta^*(\al_i)+\al_i+\al_j=0, \, a_{ij}=-1, \\
0 & \text{otherwise}.
\end{cases}
}
\item \label{lem:bb:formulas:case2}
Suppose $\theta^*(\al_i)+\al_j \in -\Phi^+ \cup \{0\}$ and $j \in I \backslash X$.
Then $i \in I \backslash X$, $\tau(i)=j$, $\theta^*(\al_i)+\al_j \in -\Phi^+_X \cup \{0\}$ and
\eq{
\ad(b_i)^{M_{ij}}(b_j) = \begin{cases}
(1+\ze_X(\al_i)\chi(\al_i-\al_j)) [\theta(f_i),f_j] \in \mfn^+_X & \text{if } \theta^*(\al_i)+\al_j <0, \, a_{ij}=0, \\
\chi(\al_j)^{-1}h_i - \chi(\al_i)^{-1}h_j & \text{if } \theta^*(\al_i)+\al_j =0, \, a_{ij}=0, \\
2(\chi(\al_i)^{-1}+\chi(\al_j)^{-1})b_i & \text{if } \theta^*(\al_i)+\al_j = 0, \, a_{ij}=-1, \\
0 & \text{otherwise}.
\end{cases}
}
\item \label{lem:bb:formulas:case3}
Suppose $\theta^*(\al_i)=-\al_i$ and $j \in I \backslash X$.
Then $i \in X^\perp$, $\tau(i)=i$ and
\eq{ \label{bb:Onsager}
\ad(b_i)^{M_{ij}}(b_j) = \sum_{r \in \Z_{>0} \atop 2r \le M_{ij}} p_{ij}^{(r,M_{ij})} \chi(\al_i)^{-r} \ad(b_i)^{M_{ij}-2r}(b_j)
}
where $p_{ij}^{(r,m)}$ for $m,r \in \Z_{\ge 0}$ such that $0 \le 2r \le m \le M_{ij}$ are negative integers defined as follows.
We set $p_{ij}^{(0,m)} = -1$ and $p_{ij}^{(r,m)} = 0$ if $2r>m$ and impose the recursion
\eq{
p_{ij}^{(r,m)} = p_{ij}^{(r,m-1)} + (m-1)(M+1-m) p_{ij}^{(r-1,m-2)} \qq \text{if } 0 < r \le \lfloor \tfrac{m}{2} \rfloor.
}
\item \label{lem:bb:formulas:case4}
If $i$ and $j$ do not satisfy any of the three conditions above then $\ad(b_i)^{M_{ij}}(b_j)=0$.
\end{enumerate}
\end{lemma}

\begin{proof}
This follows from a careful root space analysis, see \cite[App.\ A]{RV20}.
\end{proof}

The relation \eqref{bb:Onsager} and the coefficients $p_{ij}^{(r,m)}$ also appeared in \cite[Def.\ 2.3]{St19} which deals with the case $\theta=\omega$ and $\mfk = \mfg^\om$.
Such $\mfk$ are also called (embedded) \emph{generalized Onsager algebras}.

The following result is a key motivation for introducing generalized Satake diagrams.

\begin{thrm}{\cite[Thm.\ 3.3]{RV20}} \label{thm:k:GSat}
Let $(X,\tau,\chi)$ be an enriched compatible decoration and let $\theta = \Ad(\chi) \circ \theta(X,\tau)$.
The following statements are equivalent:
\begin{enumerate}
\item \label{thm:k:GSat:combinatorics}
$(X,\tau,\chi)$ is an enriched generalized Satake diagram;
\item \label{thm:k:GSat:niceSerre}
for all $i,j \in I$ such that $i \ne j$ one has
\eq{
\ad(b_i)^{M_{ij}}(b_j) \in \mfn^+_X \oplus \mfh^\theta \oplus \bigoplus_{\ell \in \Z_{>0}} \bigoplus_{\bm i \in I^\ell \atop \al_{\bm i}<\la_{ij}} \F b_{\bm i};
}
\item \label{thm:k:GSat:kdecomposition}
the following identity of $\mfh^\theta$-modules holds:
\eq{
\mfk = \mfn^+_X \oplus \mfh^\theta \oplus \bigoplus_{\bm i \in \mc J} \F b_{\bm i};
}
\item \label{thm:k:GSat:kpsfxdpt}
$\mfk$ is a pseudo-fixed-point subalgebra.
\end{enumerate}
\end{thrm}

\begin{proof}
We refer to \cite{RV20} for the proof of the equivalence of the first 3 statements and condition \eqref{defn:psfixptsubalg:1} in the definition of pseudo-fixed-point subalgebra.
It suffices to prove that \ref{thm:k:GSat:niceSerre} implies condition \eqref{defn:psfixptsubalg:2} in the definition of pseudo-fixed-point subalgebra.
In fact, since $\mfg_\al \subset \mfk$ if $\al \in \Phi_X$ and $\theta^*$ interchanges $\Phi^+ \backslash \Phi_X$ and $-\Phi^+ \backslash \Phi_X$, it suffices to prove \eqref{defn:psfixptsubalg:2} with $\al \in -\Phi^+ \backslash \Phi_X$.

First of all, by an induction argument it follows from \ref{thm:k:GSat:niceSerre} that, for all $\bm i \in \mc J$, 
\eq{
b_{\bm i} - f_{\bm i} - \theta(f_{\bm i}) \in  \mfn^+_X \oplus \mfh^\theta \oplus \bigoplus_{\bm j \in \mc J \atop \al_{\bm j}<\al_{\bm i}} \F b_{\bm j},
}
see \cite[Eq.\ (3.20)]{RV20}. 
Hence, $f_{\bm i} + \theta(f_{\bm i})$ lies in $\mfk$.
On the other hand, $f_{\bm i} + \theta(f_{\bm i})$ lies in $\mfg_{-\al_{\bm i}} + \mfg_{-\theta^*(\al_{\bm i})}$. 
Now let $\al \in \Phi^+$ be arbitrary.
Since $\{ f_{\bm i} \, : \, \bm i \in \mc J, \, \al_{\bm i}= \al \}$ is a basis for $\mfg_{-\al}$, Lemma \ref{lem:b:triangular} implies that $\{ f_{\bm i} + \theta(f_{\bm i}) \, : \, \bm i \in \mc J, \, \al_{\bm i}= \al \}$ is a basis for $\mfg_{-\al}+\mfg_{-\theta^*(\al)}$, as required.
\end{proof}

\begin{rmk}
In \cite[Thm.\ 3.3]{RV20} we use the notation $\ga_i$ for the scalar $\chi(\al_i)^{-1}$, $i \in I \backslash X$; note~that the values $\chi(\al_j)$ for $j \in X$ are irrelevant for Theorem \ref{thm:k:GSat} and we may as well assume that $\chi|_{Q_X}=1$, cf.\ Corollary \ref{crl:pseudoinv:class:2}. \hfill \rmkend
\end{rmk}
 
%%% 3.4. 
 
\subsection{Iwasawa decomposition for pseudo-symmetric pairs}

Consider the subalgebra
\eq{
\mfn^+_\theta := \mfn^+ \cap \theta(\mfn^-) \subseteq \mfn^+.
}
Since $\mfn^+ \cap \theta(\mfn^+) = \mfn^+_X$ and $\mfh$ is $\theta$-stable, immediately we obtain $\mfn^+ = \mfn^+_X \oplus \mfn^+_\theta$.
Hence the triangular decomposition of $\mfg$ implies
\eq{
\mfg = \mfn^+_X \oplus \mfh^\theta \oplus \mfn^- \oplus \mfh^{-\theta} \oplus \mfn^+_\theta.
}
Next, $\mfn^- = \bigoplus_{\bm i \in \mc J} \F f_{\bm i}$ by definition of $\mc J$.
By Lemma \ref{lem:b:triangular} we may write
\eq{
\mfg = \mfn^+_X \oplus \mfh^\theta \oplus \bigoplus_{\bm i \in \mc J} \F b_{\bm i} \oplus \mfh^{-\theta}  \oplus \mfn^+_\theta,
}
which a priori is a decomposition of linear spaces, since $\bigoplus_{\bm i \in \mc J} \F b_{\bm i}$ is not a subalgebra of $\mfg$.
However Theorem \ref{thm:k:GSat} \ref{thm:k:GSat:kdecomposition} now implies the following decomposition for $\mfg$ in terms of subalgebras.

\begin{crl}[Iwasawa decomposition for the pseudo-symmetric pair $(\mfg,\mfk)$] \label{crl:Iwasawa}
Let $(X,\tau,\chi)$ be an enriched compatible decoration and let $\theta = \Ad(\chi) \circ \theta(X,\tau)$.
Then $\mfg$ has the following decomposition in terms of subalgebras:
\eq{
\mfg = \mfk \oplus \mfh^{-\theta}  \oplus \mfn^+_\theta
}
if and only if $(X,\tau,\chi)$ is an enriched generalized Satake diagram.
\end{crl}

%%% 3.5.

\subsection{A combinatorial description of $\mfk'$} \label{sec:kprime}

We complete this section with a description of the derived subalgebra $\mfk'$, which indicates how the universal enveloping algebra $U\mfk$ can be modified by scalar terms. 
This directly extends the discussion in \cite[Sec.\ 3.3]{RV20} to the Kac-Moody setting.
Let $(X,\tau,\chi)$ be an enriched generalized Satake diagram and fix a subset $I^* \subset I \backslash X$ intersecting each $\tau$-orbit in a singleton.
Set 
\eq{
\mc J_X := \mc J \cap \bigcup_{\ell >0} X^\ell.
}
From \eqref{theta:identity} it follows that $\{ h_i \}_{i \in X} \cup \{ h_i - h_{\tau(i)} \}_{i \in I^*, \tau(i) \ne i}$ is a basis for $\mfh^\theta$.
Hence, Theorem \ref{thm:k:GSat} implies that 
\eq{
\{ e_{\bm i} \}_{\bm i \in \mc J_X} \cup \{ h_i \}_{i \in X} \cup \{ h_i - h_{\tau(i)} \}_{i \in I^*, \tau(i) \ne i} \cup \{ b_{\bm i}  \}_{\bm i \in \mc J}
}
is an $\F$-basis for $\mfk$.
Now consider the following subsets of $I^*$:
\begin{align}
I_{\sf diff} &= \{ i \in I^* \, : \, i \notin (X \cup {\tau(i)})^\perp \land \tau(i) \ne i\} , \\
I_{\sf ns} &= \{ i \in I^* \, : \, i \in X^\perp \land \tau(i)=i \} ,  \label{Ins:def} \\
I_{\sf nsf} &= \{ j \in I_{\sf ns} \, : \, \forall i \in I_{\sf ns} \; a_{ij} \in 2\Z \}.
\end{align}

Elements of $I_{\sf diff}$ and $I_{\sf nsf}$ are called \emph{special $\tau$-orbits}.

\begin{prop}{\cite[Prop.~3.2]{RV20}}
Let $(X,\tau,\chi)$ be an enriched generalized Satake diagram.
Then the set
\eq{
\{ e_{\bm i} \}_{\bm i \in \mc J_X} \cup \{ h_i \}_{i \in X} \cup \{ h_i - h_{\tau(i)} \}_{i \in I^*\backslash I_{\sf diff}, \tau(i) \ne i} \cup \{ b_{\bm i}  \}_{\bm i \in \mc J \backslash I_{\sf nsf}}
}
is an $\F$-basis for $\mfk'$.
Moreover, 
\eq{ \label{k:basis}
\mfk \cap \mfg' = \mfk' \rtimes \Bigg( \bigoplus_{i \in I_{\sf diff}} \F (h_i-h_{\tau(i)}) \oplus \bigoplus_{j \in I_{\sf nsf}} \F b_j \Bigg).
}
\end{prop}

\begin{rmk} \mbox{}
\begin{enumerate}
\item
In \eqref{k:basis}, the intersection with $\mfg'$ is only necessary since $\mfh$ may contain additional $\theta$-fixed elements if $\cork(A)>1$. 
In other words, $\mfk \subseteq \mfg'$ if $\cork(A) \le 1$.
\item
By Definition \ref{def:GSat} \ref{def:GSat:enriched}, if $(X,\tau,\chi)$ is an enriched generalized Satake diagram and $i \in I^*$ then $\chi(\al_i)$ and $\chi(\al_{\tau(i)})$ are allowed to be \emph{different} precisely if $i \in I_{\sf diff}$. This explains the notation.
\item
The set $I_{\sf ns}$ is implicit in Lemma \ref{lem:bb:formulas} \ref{lem:bb:formulas:case3}.
Its elements are called \emph{nonstandard}; this nomenclature goes back to \cite[Sec.~7, Variation 2]{Le02} and \cite[Sec.~5]{Ko14} where it was pointed out that if $i \in I_{\sf nsf}$ then a \emph{free} parameter appears in the extra Cartan term in the quantum analogues of the generators $b_i$ in $U_q\mfk$.
The non-deformed version of this statement is the result \cite[Prop.~3.13]{RV20}, which also extends directly to the Kac-Moody setting.
It characterizes, in terms of the set $I_{\sf nsf}$, the essential freedom in the enveloping algebra $U\mfk$ with respect to the additional scalar terms. 
\hfill \rmkend
\end{enumerate}
\end{rmk}

%%%%%%%%%%%%%%%%%%%%%%%%%%%%%%%%%%%%%%%%%%%%%%%%%%%%%%%%%%%%%%%%%%
% EoF
%%%%%%%%%%%%%%%%%%%%%%%%%%%%%%%%%%%%%%%%%%%%%%%%%%%%%%%%%%%%%%%%%%

%%%%%%%%%%%%%%%%%%%%%%%%%%%%%%%%%%%%%%%%%%%%%%%%%%%%%%%%%%%%%%%%%%
% Section 4
%%%%%%%%%%%%%%%%%%%%%%%%%%%%%%%%%%%%%%%%%%%%%%%%%%%%%%%%%%%%%%%%%%

\section{The restricted Weyl group and restricted root system} \label{sec:restricted}

In this section we give a survey on results on the restricted Weyl group and restricted root system associated to a root system involution of a Kac-Moody algebra, thereby generalizing some of Heck's work \cite{He84} for finite root systems and drawing on works by Lusztig \cite{Lu95} (also see \cite[Ch.~25]{Lu03}) and Geck and Iancu \cite{GI14}. 

For finite root systems $\Phi$, Heck \cite{He84} studied involutive automorphisms of finite root systems $\Phi$. 
He obtained a diagrammatic constraint on involutions of finite root systems such that the associated restricted root system and restricted Weyl group have natural properties, thereby essentially obtaining the notion of a generalized Satake diagram (although the terminology in \cite{He84} is different). 
Heck's work was a simplification of Schattschneider's work \cite{Sch69} in the still more general case of an arbitrary group of automorphisms of a finite root system $\Phi$.

Note that the Weyl group $W$ is a Coxeter group (see e.g.\ \cite[\S 3.13]{Ka90}, \cite{MT72} and \cite{Lo80}), acting by reflections in the root system of $\mfg$.
A particular obstacle in the restricted setting is that, given a compatible decoration $(X,\tau)$ and the associated root system involution $w_X \circ \tau$, there are no fewer than three different natural definitions of the restricted Weyl group, which we will denote by $W(\ol \Phi)$, $\ol W$ and $\wt W$.
We will see that, precisely if $(X,\tau)$ is a generalized Satake diagram, these three groups are isomorphic.
Moreover in this case this group is a Coxeter group and the restricted root system is stable under its action.

%%% 4.1.

\subsection{The $\Q$-span of the root system}

By \eqref{Autgh:decomposition}, the subgroup of $\GL(\mfh')$ of invertible linear maps on $\mfh'$ that extend to Lie algebra automorphisms of $\mfg$ is of the form $\Out(A) \ltimes W$.
The definition of root space implies that the duals of such maps on $\mfh'$ necessarily stabilize $\Phi$.
Conversely, an invertible linear map on the dual of $\mfh'$ which stabilizes $\Phi$ can be dualized and extended to a Lie algebra automorphism stabilizing $\mfh$.

It is convenient here to work over the ordered subfield $\Q \subset \F$ in order to facilitate certain (geometric) constructions, so we consider the following $|I|$-dimensional vector space over $\Q$:
\eq{
V := \Sp_\Q \Phi = \Sp_\Q \Pi.
}
The group of \emph{root system automorphisms} $\Aut(\Phi) \cong \Out(A) \ltimes W$ is the subgroup of those $\si \in \GL(V)$ which stabilize $\Phi$.
Note that the bilinear form $(\phantom{x},\phantom{x})$ restricts to an $\Aut(\Phi)$-invariant $\Q$-valued bilinear form on $V$.

%%% 4.2.

\subsection{Root system involutions and the corresponding orthogonal decompositions}

By the results of Section \ref{sec:pseudoinvolutions}, any pseudo-involution of the second kind is $\Inn(\mfg)$-conjugate to $\Ad(\chi) \circ \theta(X,\tau)$ for some extended compatible decoration $(X,\tau,\chi)$.
The automorphism $\Ad(\chi) \circ \theta(X,\tau)$ restricts to the linear involution $-w_X \circ \tau$ on $\mfh$, stabilizing the subspace $\mfh'$.
The dual map on $(\mfh')^* = \Sp_\F\Phi$ is given by the same formula.
It restricts to the $\Q$-span $V$ of $\Phi$ and is an example of an involutive root system automorphism, which we call \emph{root system involution}.

Our focus will be on the map
\eq{
\si := w_X \circ \tau = -\theta(X,\tau)^*|_V \in \Aut(A) \ltimes W,
}
which is also the convention used in \cite{He84}.
Since $\si$ is an involution the following orthogonal linear decomposition holds:
\eq{ \label{V:decomposition}
V = V^\si \oplus V^{-\si}.
}
We denote by $\ol{\phantom{x}}$ the corresponding projection: $V \to V^\si$, so that 
\eq{
\ol{\la} = \frac{\la + \si(\la)}{2} \qq \text{for all } \la \in V.
}
Because $\si$ is an isometry, the following property holds, which we will repeatedly use in this section:
\eq{
(\ol{\be},\ol{\ga}) = (\ol{\be},\ga) = (\be,\ol{\ga}) \qq \text{for all } \be,\ga \in V.
}

For $J \subseteq I$ we denote $V_J := \Sp_\Q \Phi_J$.
Since $(X,\tau)$ is a compatible decoration, the identity $\si|_{V_X} = w_X \circ \oi_X|_{V_X} = -\id_{V_X}$ holds and hence the projection $\ol{\phantom{x}}$ annihilates $V_X$.
The linear map $w_X - \id$ maps $V$ to $V_X$; since $\si$ commutes with $w_X$, we obtain 
\eq{ \label{bar:wX}
w_X \circ \ol{\phantom{x}} = \ol{\phantom{x}} \circ w_X = \ol{\phantom{x}}.
}
By multiplying by $\si$ we also obtain
\eq{ \label{bar:tau}
\tau \circ \ol{\phantom{x}} = \ol{\phantom{x}} \circ \tau = \ol{\phantom{x}}.
}
From \eqrefs{bar:wX}{bar:tau} it follows that $V^\si \subseteq V^{w_X}$ and $V^\si \subseteq V^\tau$.
Since $V^{w_X} \cap V^{\tau} \subseteq V^\si$ is clear, we obtain
\eq{ \label{Vsigma:Vtau:VwX}
V^\si = V^{w_X} \cap V^\tau.
}

%%% 4.3.

\subsection{The restricted root system}

\begin{defn}
Let $\si$ be a root system involution.
For any subset $\Si \subseteq \Phi$ (not necessarily a root subsystem), we denote
\eq{
\ol \Si := \{ \ol{\al} \, : \, \al \in \Si \} \backslash \{0\}.
}
We call $\ol \Phi$ the \emph{restricted} (or \emph{relative}) \emph{root system} associated to $\si$ and its elements \emph{restricted roots}; the dimension of $\Sp_\Q(\ol\Phi)$ is called the \emph{restricted rank} of $\si$. \hfill \defnend
\end{defn}

If $\si = w_X \circ \tau$ with $(X,\tau)$ a compatible decoration, then $\ol{\al}=0$ if and only if $\al \in \Phi_X$, so that 
\eq{ \label{barSigma:X}
\ol \Si = \{ \ol{\al} \, : \, \al \in \Si \backslash \Phi_X \}
}
for any $\Si \subseteq \Phi$.
Equivalently, $\ol\Phi$ may be defined as $\{ \al|_{\mfh^{-\theta(X,\tau)}} \} \backslash \{0\}$, see e.g.~\cite[Sec.\ 2.3]{DK19}.

The restricted root system can be empty, non-reduced and non-crystallographic.
More problematically, it is not necessarily stable under reflections in hyperplanes orthogonal to its elements.
We will return to this question shortly, indicating constraints on $(X,\tau)$ to guarantee that $\ol\Phi$ is reflection-stable.

\begin{lemma} \label{lem:restrictedsize}
Let $\si$ be a root system involution and let $J \subseteq I$ be such that $\si$ stabilizes $V_J$.
\begin{enumerate}
\item \label{lem:restrictedsize:finite}
If $J$ is of finite type, then $\ol{\Phi_J}$ is finite and, since $(V_J)^\si \subseteq V_J$, the restriction of $(\phantom{x},\phantom{x})$ to $(V_J)^\si$ is positive definite.
\item \label{lem:restrictedsize:infinite}
If $J$ is of infinite type we denote by $\Phi^+_{\im,J}$ the set of positive imaginary roots in $\Phi_J$.
Then $\ol{\Phi^+_{\im,J}}$ is infinite and contained in $\ol{\Phi^+_J}$.
Furthermore, if $(V_J)^\si$ is 1-dimensional then $(\be,\be) \le 0$ for all $\be \in \ol{\Phi_J}$.
\end{enumerate}
\end{lemma}

\begin{proof}\mbox{}
\begin{enumerate}[leftmargin=!,font={\itshape},labelwidth=\widthof{\emph{(ii)}}]
\item[{\ref{lem:restrictedsize:finite}}]
If $J$ is of finite type, then $\Phi_J$ is finite and hence $\ol{\Phi_J}$ is finite.
Since the restriction of the bilinear form $(\phantom{x},\phantom{x})$ to $V_J = \Sp_\Q(\Phi_J)$ is positive definite and $\si$ is an isometry, the bilinear form $(\phantom{x},\phantom{x})$ restricts to $(V_J)^{\si}$ and the restriction is positive definite.
\item[{\ref{lem:restrictedsize:infinite}}] If $J$ is of infinite type then by \cite[Thm.\ 5.6]{Ka90}, $\Phi^+_{\im,J}$ is nonempty.
Note that $\Aut_{J}(A) \ltimes W_{J}$ stabilizes $\Phi^+_{\im,J}$ as a consequence of \cite[Thm.\ 5.4]{Ka90}.
In particular, $\si$ stabilizes $\Phi^+_{\im,J}$ and therefore $\ol{\la} \ne 0$ if $\la \in \Phi^+_{\im,J}$.
It follows that $\ol{\Phi^+_{\im,J}} \subseteq \ol\Phi$.
By \cite[Prop.\ 5.5]{Ka90}, $\Phi^+_{\im,J}$ is stable under multiplication by $\Z_{\ge 0}$, so that $\ol{\Phi^+_{\im,J}}$ is infinite.
Finally, by \cite[Thm.\ 5.6]{Ka90} there exists $\la \in \Phi_{\im,J}^+$ such that $(\la, \mu) \le 0$ for all $\mu \in \Phi^+$.
Hence $(\la,\la) \le 0$ and moreover, since $\si(\la) \in \Phi^+$, also $(\la,\si(\la)) \le 0$.
Therefore $(\ol{\la},\ol{\la}) = (\la,\ol{\la}) \le 0$ as required. 
If $(V_J)^\si$ is 1-dimensional, it is spanned over $\Q$ by $\ol{\la}$ and we obtain $(\be,\be) \le 0$ for all $\be \in \ol{\Phi_J}$. \qedhere
\end{enumerate}
\end{proof}

\begin{rmk}
If $J$ is of affine type then we can strengthen Lemma \ref{lem:restrictedsize} \ref{lem:restrictedsize:infinite} as a consequence of \cite[Theorem 5.6 (b)]{Ka90}.
Namely, the kernel of $A_J$ is spanned by a unique tuple of setwise coprime positive integers $(a_j)_{j \in J}$ and $\Phi^+_{\im,J} = \Sp_{\Z_{> 0}} \del_{J}$ where $\del_{J} = \sum_{j \in J} a_j \al_j$ is the \emph{basic imaginary root}.
For all $j \in J$ it follows that $\del_J(h_j)=0$ and hence $s_j(\del_J)=\del_j$.
Since $\Aut(A_{J})$ preserves the height function on $Q_{J}$, we obtain $\tau(\del_{J})=\del_{J}$.
It follows that $\ol{\del_{J}} = \del_{J}$ and hence $\ol{\Phi^+_{\im,J}} = \Phi^+_{\im,J}$.
\hfill \rmkend
\end{rmk}

%%% 4.4.

\subsection{Combinatorial bases for $V^\si$ and $V^{-\si}$.}

Recall that we have fixed a subset $I^* \subset I \backslash X$ which intersects each $\tau$-orbit in a singleton.
Using \eqref{bar:tau} and \eqref{barSigma:X}, we obtain 
\eq{
\ol \Pi = \{ \ol{\al_i} \}_{i \in I} \backslash \{0\} = \{ \ol{\al_i} \}_{i \in I^*}.
}
Note that the set $\ol\Pi$ is linearly independent (for all $i \in I^*$, $\ol{\al_i}$ is the only element of $\ol\Pi$ such that when expanding with respect to the basis $\Pi$ of $V$, the coefficient of $\al_i$ is nonzero).

\begin{lemma} \label{lem:Vsigma:base}
Given a compatible decoration $(X,\tau)$, let $\si = w_X \circ \tau$ be the corresponding root system involution. 
Then 
\eq{
V^\si = \Sp_\Q\big( \ol \Pi \big) = \Sp_\Q\big( \ol \Phi \big) .
}
Moreover, $\ol \Phi = \ol{\Phi^+} \cup (-\ol{\Phi^+})$ with $\ol{\Phi^+}$ contained in $\Sp_{\Z_{\ge 0}}\big( \ol \Pi \big)$.
\end{lemma}

\begin{proof}
Let $\be = \sum_{i \in I} m_i \al_i \in V^\si$ with $m_i \in \Q$.
Since $\be$ is fixed by $\si$ it equals $\ol\be$ given by
\eq{ \label{sigmafixed:decomposition}
\ol{\be} = \sum_{i \in I} m_i \ol{\al_i} = \sum_{i \in I^* \atop \tau(i)=i} m_i \ol{\al_i} + \sum_{i \in I^* \atop \tau(i) \ne i} (m_i+m_{\tau(i)}) \ol{\al_i}.
}
Therefore $V^\si \subseteq \Sp_\Q(\ol \Pi)$.
On the other hand, one has the natural inclusions $\Sp_\Q\big(\ol{\Pi}\big) \subseteq \Sp_\Q\big( \ol \Phi \big) \subseteq V^\si$ and the first statement follows.
To establish the second statement, suppose $\be \in \ol \Phi$. 
Then $\be = \ol{\pm \la}$ for some $\la \in \Phi^+$.
Hence either $\be$ or $-\be$ lies in $\sum_{i \in I} m_i \al_i$ with $m_i\in \Z_{\ge 0}$.
Now \eqref{sigmafixed:decomposition} implies that $\pm \be$ is a $\Z_{\ge 0}$-linear combination of the $\ol{\al_i}$.
\end{proof}

\begin{rmk}
If $B \subseteq R \subseteq V$ is an inclusion of sets with $V$ a $\Q$-linear space spanned by $R$, $B$ is called a \emph{base} for $R$ if $B$ is a $\Q$-basis for $V$ and $R \subseteq \Sp_{\Z_{\ge 0}}B \cup \Sp_{\Z_{\le 0}}B$.
The root space decomposition \eqref{g:rootspacedecomposition} implies that $\Pi$ is a base for $\Phi$. 
Owing to Lemma \ref{lem:Vsigma:base}, $\ol \Pi$ is a base for $\ol \Phi$. \hfill \rmkend
\end{rmk}

It follows from Lemma \ref{lem:Vsigma:base} that the restricted rank of $\si$ is the number of $\tau$-orbits outside $X$; hence we will also call this number the restricted rank of $(X,\tau)$.
We refer to the elements $\ol{\al_i}$ as \emph{simple restricted roots}.
For all $\be \in \ol\Phi^+$ there exist unique $m_i \in \Z_{\ge 0}$ such that $\be = \sum_{i \in I^*} m_i \ol{\al_i}$ and we call $\sum_{i \in I^*} m_i \in \Z_{\ge 0}$ the \emph{restricted height} of $\be$.
Note that $\be$ has restricted height 1 if and only if $\be = \ol{\al_i}$ for some $i \in I^*$.

Knowledge of the signs of the inner products between simple restricted roots will allow us to study the action of the reflections with respect to the $\ol{\al_i}$ on $\ol\Phi$.
For this purpose, recall the notation $X[i] = X \cup \{ i, \tau(i) \}$ for $i \in I \backslash X$.
It suffices to take $i \in I^*$.
It will turn out to matter greatly whether $X[i]$ is of finite type, so we define
\eq{
\wt I = \{ i \in I^* \, : \, X[i] \text{ is of finite type} \}.
}

\begin{exam}
Let $(X,\tau)$ be a compatible decoration.
If $I$ is of finite type then $\wt I = I^*$.
If $I$ is of affine type then $I^*$ is nonempty (since $X \ne I$) and there are two possibilities:
if $|I^*|=1$ then $X[i]=I$ for $i \in I^*$ and hence $\wt I = \emptyset$; if $|I^*|> 1$ then $\wt I = I^*$. \hfill \examend
\end{exam}

\begin{lemma} \label{lem:innerproducts}
Let $i,j \in I^*$.
Then $(\ol{\al_i},\ol{\al_j}) > 0$ if and only if $i=j \in \wt I$.
\end{lemma}

\begin{proof}
We split up the proof into some casework.
We will repeatedly use that $\si$ stabilizes $V_{X[i]}$ for any $i \in I^*$ and is of restricted rank 1 as a root system involution of $V_{X[i]}$.
\begin{description}
\item[$i=j \in \wt I$]
In this case $X[i]$ is of finite type and by Lemma \ref{lem:restrictedsize} \ref{lem:restrictedsize:finite} we obtain $(\ol{\al_i},\ol{\al_i}) > 0$.
\item[$i=j \notin \wt I$]
Now $X[i]$ is of infinite type and $(\ol{\al_i},\ol{\al_j}) \le 0$ by Lemma \ref{lem:restrictedsize} \ref{lem:restrictedsize:infinite}.
\item[$i \ne j$]
In this case $i \notin X[j]$.
Recalling Lemma \ref{lem:Xtau:basics} \ref{lem:Xtau:basics:wX:formula} one has $\si(\al_j) = \al_{\tau(j)} + \sum_{j \in X} v_{ij} \al_j$ with $v_{ij} \in \Z_{\ge 0}$.
Hence
\eq{
2(\ol{\al_i},\ol{\al_j}) = (\al_i, \al_j + \si(\al_j)) = (\al_i,\al_j)+(\al_i,\al_{\tau(j)})+\sum_{k \in X} v_{ij} (\al_i,\al_k) \le 0,
}
as required. \qedhere
\end{description}
\end{proof}

It follows that, if $\wt I$ is a proper subset of $I^*$, there are imaginary simple restricted roots, as there are for Borcherds' generalized Kac-Moody algebras \cite{Bo88}. 

We close this subsection with a description of the orthogonal complement $V^{-\si}$.

\begin{lemma} \label{lem:Vtheta}
Given a compatible decoration $(X,\tau)$, let $\si = w_X \circ \tau$ be the corresponding root system involution. 
Then
\eq{
V^{-\si} = \Sp_\Q(\{ \al_i \}_{i \in X} \cup \{ \al_i - \al_{\tau(i)} \, : \, \tau(i) \ne i \}_{i \in I^*})
}
and $\Phi \cap V^{-\si} = \Phi_X$.
\end{lemma}

\begin{proof}
Let $\be \in V^{-\si}$.
Then there exist $m_i \in \Q$ such that $\be = \sum_{i \in I} m_i \al_i$. 
It follows that $\be$ equals
\eq{
\frac{\be-\si(\be)}{2} = \sum_{i \in I} m_i \frac{\al_i-\si(\al_i)}{2} = \sum_{i \in X} m_i \al_i + \sum_{i \in I \backslash X} \frac{m_i}{2} (\al_i - w_X(\al_{\tau(i)})).
}
Here, the element $\al_i - w_X(\al_{\tau(i)})$ equals $\al_i-\al_{\tau(i)}$ plus a term in $\Sp_\Q\{ \al_i \}_{i \in X}$, so we obtain the first statement.
The second statement follows from the fact that $\Phi = \Phi^+ \cup (-\Phi^+)$.
\end{proof}

%%% 4.5.

\subsection{The Weyl group of the restricted root system}

Let $\be \in V$ such that $(\be,\be)>0$.
Consider the idempotent linear map $p_\be:V \to V$ defined by $p_\be(\ga) = \frac{(\be,\ga)}{(\be,\be)} \be$ for all $\ga \in V$; it is idempotent.
Consider the involution
\eq{
s_\be := \id - 2p_\be.
}
Note that the orthogonal decomposition $V = \Q \be \oplus \{ \la \in V \, : \, (\la,\be)=0 \}$ diagonalizes $s_\be$ and we refer to $s_\be$ as the \emph{orthogonal reflection} associated to $\be$, respectively.
Because $p_\be$ is self-adjoint, so is $s_\be$; we conclude that $s_\be$ is an isometry. 
Finally, $s_\ga=s_\be$ for all $\ga \in \Q_{\ne 0}\be$.

If in addition $\be \in V^\si$ for some root system involution $\si$ then one easily checks that $p_\be$ commutes with $\si$, so that $s_\be$ commutes with $\si$.
Hence $s_\be$ stabilizes $V^\si$ and is therefore an isometry of $V^\si$.
If $f$ is any isometry of $V^\si$ then
\eq{ \label{restrictedreflection:isometry}
f \circ s_\be \circ f^{-1} = s_{f(\be)}.
}

\begin{defn}
The \emph{Weyl group of $\ol\Phi$} is the following group:
\eq{
W(\ol{\Phi}) := \big\langle \{ s_\be|_{V^\si} \, : \, \be \in \ol\Phi, \, (\be,\be)>0 \} \big\rangle \le \GL(V^\si).
}
We call the elements $\ol{s_i} := s_{\ol{\al_i}}|_{V^\si}$ with $i \in \wt I$ \emph{simple restricted reflections}. \hfill \defnend
\end{defn}

The following basic property will later help us show that the set $\{ \ol{s_i} \}_{i \in \wt I}$ generates $W(\ol\Phi)$.

\begin{lemma} \label{lem:simpleresrefl:stable}
For all $i \in \wt I$, $\ol{s_i}$ stabilizes $\ol{\Phi^+} \backslash \Q \ol{\al_i}$.
\end{lemma}

\begin{proof}
This is a variation on a standard argument, see e.g.~\cite[Prop.\ 1.4]{Hu90}.
Let $i \in \wt I$ be arbitrary and let $\be \in \ol{\Phi^+} \backslash \Q \ol{\al_i}$.
By Lemma \ref{lem:Vsigma:base}, there exist $m_j \in \Z_{\ge 0}$ such that
\eq{
\be = \sum_{j \in I^*} m_j \ol{\al_j}.
}
Since $\be \notin \Q \ol{\al_i}$ there exists $j \in I^* \backslash \{ i \}$ such that $m_j \in \Z_{>0}$.
At the same time $\ol{s_i}(\be) = \be - 2 \frac{(\ol{\al_i},\be)}{(\ol{\al_i},\ol{\al_i})} \ol{\al_i}$ with the same coefficient $m_j$.
By the second statement of Lemma \ref{lem:Vsigma:base}, $\ol{s_i}(\be) \in \ol{\Phi^+}$.
Finally, if $\ol{s_i}(\be) \in \Q \ol{\al_i}$ then applying $\ol{s_i}$ produces a contradiction with $\be \notin \Q \ol{\al_i}$.
\end{proof}

%%% 4.6.

\subsection{The group $W^\si$ and the restricted Weyl group $\ol W$}

For any linear map $\psi: V \to V$ and any subset $S \subseteq W$ we may consider the following subset of $W$:
\eq{
S^\psi = \{ w \in S \, : \, w \circ \psi = \psi \circ w \text{ as linear maps on } V \}.
}
Given a root system automorphism $\si$, for all $w \in W^\si$, $w$ stabilizes $V^\si$ and we may consider $w|_{V^\si}$.

\begin{defn}[{\cite{Sch69,He84}}]
Given a root system involution $\si$ of $\Phi$, we define the \emph{restricted Weyl group} as
\begin{flalign*}
&& \ol W := \{ w|_{V^\si} \, : \, w \in W^\si \}. && \hfill \defnend
\end{flalign*}
\end{defn}

We may consider the groups $W(\ol \Phi)$ and $\ol W$ for any root system involution.
It is natural to expect that for suitable $\si$ they are isomorphic. 
In the remainder of this section we will always assume 
\eq{
\si = w_X \circ \tau
}
for some compatible decoration $(X,\tau)$.
We will show in the remainder of this section that for certain compatible decorations both $W(\ol \Phi)$ and $\ol W$ are isomorphic to a third combinatorially defined group $\wt W \le W^\si$ (and hence to each other, as desired).
In order to do this we develop some more properties of $W^\si$ and $\ol W$.
Note that
\eq{ \label{Wsigma:alt}
W^\si = \{ w \in W \, : \, w_X w = \tau(w) w_X \}
}
and that any $w \in W_X$ satisfies $w_X w w_X = \oi_X(w) = \tau(w)$, so that $W_X \le W^\si$.

\begin{lemma} \label{lem:wX:kernel}
The group $W_X$ is the kernel of the homomorphism $\cdot|_{V_\si} : W^\si \to \ol W$.
\end{lemma}

\begin{proof}
The proof of \cite[Prop.~3.1]{He84}, which relies on \cite[Ch.~V, \S 3.3, Prop.~1]{Bo68} and Lemma \ref{lem:Vtheta}, immediately generalizes to the Kac-Moody setting.
Also cf. \cite[Thm. 1.3]{Lo80}.
\end{proof}

As a consequence of Lemma \ref{lem:wX:kernel}, $W_X$ is a normal subgroup of $W^\si$ and 
\eq{ \label{Wsigma:semidirect}
W^\si \cong W_X \rtimes \ol W.
}
Consider the set of \emph{minimal left coset representatives} of $W_X$:
\eq{
W^X = \{ w \in W \, : \, \forall j \in X \; \ell(s_j w) > \ell(w) \}.
}
Because $\tau$ permutes $\{ s_j \}_{j \in X}$ and preserves the length function on $W$, the set $W^X$ is also $\tau$-stable.
By the natural generalization of the arguments in e.g.~\cite[Prop. 1.10]{Hu90} to infinite Coxeter groups, one has
\eq{ \label{W:cosetX:factorization}
W = W_X \cdot W^X, \qq W_X \cap W^X = \{ 1 \},
}
i.e. for all $w \in W$ there exists a unique $(u,v) \in W_X \times W^X$ such that $w=uv$.
Since $\si$ fixes $W_X$ pointwise, we deduce
\eq{ \label{Wsigma:cosetX:factorization}
W^\si = W_X \cdot (W^X)^\si, \qq W_X \cap (W^X)^\si = \{ 1 \}.
}
For two subsets $S,S' \subseteq W$, we write $N_S(S') = \{ w \in S \, | \, wS' = S'w \}$ for the normalizer of $S'$ in $S$.
We will now show that $(W^X)^\si$ can be identified with the $\tau$-fixed point subset in 
\eq{
\mc W := N_{W^X}(W_X).
}

\begin{lemma} \label{lem:Wtau:cosetX}
$(W^X)^\si = \mc W^\tau$.
\end{lemma}

\begin{proof}
Let $w \in (W^X)^\si$ be arbitrary.
Because $w$ commutes with $\si$ and $W_X$ is normal,
\eq{
w_X \cdot \tau(w) = w \cdot w_X = u \cdot w
}
for some $u \in W_X$.
Since decompositions in $W = W_X \cdot W^X$ are unique, it follows that $\tau(w)=w$ so that $(W^X)^\si \subseteq W^\tau$.
Since $W_X$ is a normal subgroup of $W^\si$ we obtain $(W^X)^\si \subseteq \mc W^\tau$.

To prove the reverse inclusion, assume that $w \in \mc W^\tau$ so that $w W_X w^{-1} = W_X$.
For all $u \in W_X$ there exists a unique $v \in W_X$ such that $wu=vw$; since $w \in W^X$ we obtain $\ell(v) \le \ell(u)$.
On the other hand, $W_X$ is finite so conjugation by $w$ is surjective. 
It follows that $ww_X=w_Xw$.
Since $\tau(w)=w$, \eqref{Wsigma:alt} now implies that $w \in W^\si$.
\end{proof}

By a statement in \cite[25.1]{Lu03}, which relies on \cite[Prop. 1.5, Lemma 2.2]{Lu03}, the set $\mc W$ is a group.
Comparing \eqref{Wsigma:semidirect} and \eqref{Wsigma:cosetX:factorization}, making use of Lemma \ref{lem:Wtau:cosetX} we obtain

\begin{prop} \label{prop:barW}
$\ol W \cong \mc W^\tau$.
\end{prop}

%%% 4.7.

\subsection{A combinatorial prescription of the simple restricted reflections: the group $\wt W$}

Recall that $W(\ol\Phi)$ was not defined as a subgroup of $W$.
We now discuss a particular subgroup $\wt W \le W$ defined in terms of explicit generators $\wt s_i$, see e.g.~\cite[\S 5]{Lu76}, \cite[Ch.~25]{Lu03} and \cite{GI14}, with the aim of having an isomorphic copy of $W(\ol\Phi)$ in $W$.
In particular we will strive to choose the $\wt s_i \in W$ so that $\wt s_i|_{V^\si} = \ol{s_i}$.
Having achieved this, we will obtain a straightforward proof that $W(\ol\Phi)$ is generated by the $\ol{s_i}$.

Note that the simple restricted reflection $\ol{s_i}$ fixes pointwise $\{ \be \in V^\si \, : \, (\be,\ol{\al_i}) =0 \}$.
The following lemma justifies selecting $\wt s_i$ from $W_{X[i]}$.

\begin{lemma} \label{lem:WextendedXfix}
Let $i \in I^*$.
If $w \in W_{X[i]}$ stabilizes $V^\si$ then $w$ fixes pointwise $\{ \be \in V^\si \, : \, (\be,\ol{\al_i}) =0 \}$.
\end{lemma}

\begin{proof}
Recall that $W_{X[i]}$ is generated by $\{ s_j \}_{j \in X} \cup \{ s_i, s_{\tau(i)} \}$.
If $j \in X$ then $\ol{\al_j} =0$ so that the orthogonality of the decomposition \eqref{V:decomposition} implies that $s_j$ fixes $V^\si$ pointwise (and hence also the relevant subspace).
Now let $\be \in V^\si$ such that $(\be,\ol{\al_i})=0$.
It remains to prove that $s_i$ and $s_{\tau(i)}$ fix $\be$. 
From \eqref{Vsigma:Vtau:VwX} it follows that $\tau(\be)=\be$.
Hence $(\be,\ol{\al_i})=0$ implies $(\be,\al_i)=(\be,\al_{\tau(i)})=0$ so that both $s_i$ and $s_{\tau(i)}$ fix $\be$.
\end{proof}

Recall that for $i \in I^* \backslash \wt I$ we have not defined a simple restricted reflection $\ol{s_i}$.
Indeed, the following lemma shows that there are no involutions in $W$ that negate $\ol{\al_i}$.

\begin{lemma} \label{lem:X[i]infinite}
Suppose that $i \in I^*$ is such that $X[i]$ is of infinite type.
No involutive element of $W$ sends $\ol{\al_i}$ to $-\ol{\al_i}$.
\end{lemma}

\begin{proof}
Lemma \ref{lem:Vsigma:base} implies that $(V_{X[i]})^\si$ is 1-dimensional, so that $(\ol{\al_i},\ol{\al_i})\le 0$ by Lemma \ref{lem:restrictedsize} \ref{lem:restrictedsize:infinite}.
Now suppose there is an involution $w \in W$ such that $w(\ol{\al_i})=-\ol{\al_i}$.
By \cite[Thm.~A(a)]{Ri82}\footnote{This is a special case of the result \cite[Prop.\ 3.3]{Sp85} which we used in Section \ref{sec:cdec}.} there exists $v \in W$ and $Z \subseteq I$ of finite type such that $w = v \circ w_Z \circ v^{-1}$. 
Hence $v^{-1}(\ol{\al_i})$ is fixed by $-w_Z$.
Since $W_Z$ maps $Q^+ \backslash Q_Z$ to $Q^+$, it follows that $v^{-1}(\ol{\al_i}) \in V_Z$.
Since $Z$ is of finite type, the bilinear form $(\phantom{x},\phantom{x})$ is positive definite on $V_Z$, so that $(\ol{\al_i},\ol{\al_i}) = (v^{-1}(\ol{\al_i}) ,v^{-1}(\ol{\al_i}) )>0$, contradicting the first statement of the Lemma.
\end{proof}

Given $i \in I^*$, we now engineer the desired element $\wt s_i$.
There are two clear desirable properties: $\wt s_i$ should commute with $\si$ (so that it stabilizes $V^\si$) and should send $\ol{\al_i} \in V_{X[i]}$ to $-\ol{\al_i}$ (hence $\wt s_i$ will be involutive).
By Lemma \ref{lem:WextendedXfix}, it is natural to choose $\wt s_i \in W_{X[i]}$.
Since the longest element acts as multiplication by $-1$ on $V_{X[i]}$ up to a diagram automorphism, it is natural to set $\wt s_i = u \cdot w_{X[i]}$ for some $u \in W_{X[i]}$.
The two required properties are now seen to be equivalent to 
\eq{ \label{conditionsontildesi:2}
\al_i+\si(\al_i) \in (Q_{X[i]})^{u \circ \oi_{X[i]}},  \qq u \cdot w_{\oi_{X[i]}(X)} = w_X \cdot \tau(u).
}
If we choose $u \in W_X$ then $u \circ \oi_{X[i]}$ maps $\al_i+\si(\al_i) \in Q^+ \backslash Q_X$ to $Q^+ \backslash Q_X$, which supports the first condition in \eqref{conditionsontildesi:2}.
A natural choice for $u$ is now $u=w_X$, since in this case $\wt s_i$ has minimal length.
For this choice the last condition in \eqref{conditionsontildesi:2} amounts to $\oi_{X[i]}(w_X) = w_X$.
Assuming this and recalling Lemma \ref{wX:formula}, we see that the first condition in \eqref{conditionsontildesi:2} is satisfied.
Hence we are led to the following definition.

\begin{defn}[{See e.g.\ \cite[25.1]{Lu03} and \cite[Rmk.~8]{GI14}}]
We define $\wt W:=\langle \{ \wt s_i \}_{i \in \wt I} \rangle \le W$ where
\begin{flalign} 
\label{tildesi:def}
&& \wt s_i := w_X \cdot w_{X[i]} \in W_{X[i]}, \qq i \in \wt I. &&\hfill \defnend
\end{flalign} 
\end{defn}

Since the $\wt s_i$ are fixed by $\tau$ in $\Aut(A) \ltimes W$, we immediately obtain
\eq{ \label{tildeW:taufixed}
\wt W \le W^\tau.
}

\begin{exam} 
It is useful to spell out the generator $\wt s_i$ in certain typical configurations.
Let $i \in \wt I$.
\begin{enumerate} 
\item
If $i \in X^\perp$ then $\tau(i) \in X^\perp$ by Lemma \ref{lem:Xtau:basics} \ref{lem:Xtau:basics:connected}.
In this case $w_{X[i]} = w_X \cdot w_{\{ i,\tau(i)\}}$ so that $\wt s_i = w_{\{ i,\tau(i) \}}$.
There are now three possibilities: $\tau(i)=i$ (i.e. $i \in I_{\sf ns}$), so that $\{i,\tau(i)\}$ is of type ${\sf A}_1$ and $\wt s_i = s_i$; $a_{i \tau(i)}=0$ (i.e. $\tau(i) \ne i$ and $i \notin I_{\sf diff}$), so that $\{i,\tau(i)\}$ is of type ${\sf A}_1 \times {\sf A}_1$ and $\wt s_i = s_i s_{\tau(i)}$; $a_{i\tau(i)}=1$, so that $\{i,\tau(i)\}$ is of type ${\sf A}_2$ and $\wt s_i = s_is_{\tau(i)}s_i$.
\item
Suppose that $\tau(i)=i$ and that $X[i]$ is of rank 2.
Hence $X=\{j\}$ and $a_{ij}a_{ji} \in \{0,1,2,3\}$.
It follows that $\wt s_i = w_{\{ j\}} \cdot w_{\{i,j\}} \in \{ s_i, s_i s_j, s_i s_j s_i, s_i s_j s_i s_j s_i \}$, respectively.
Note that $\wt s_i$ is an involution except if $a_{ij}a_{ji}=1$. 
\hfill \examend
\end{enumerate}
\end{exam}

The following key result shows that, given the definition \eqref{tildesi:def}, the desirable properties of $\wt s_i$ are equivalent \emph{to each other} and to the condition that $(X,\tau)$ is a generalized Satake diagram.
It extends the result \cite[Thm.\ 2.1]{RV20} which itself combined and extended results in \cite{He84} and \cite{Lu03}.

\begin{prop} \label{prop:tildeW}
The following statements are equivalent:
\begin{enumerate}
\item \label{prop:tildesi:commutesigma}
$\wt W \le W^\si$; 
\item \label{prop:tildesi:involution}
$\wt s_i$ is an involution for all $i \in \wt I$;
\item \label{prop:tildesi:preserveX}
$\{ s_j \}_{j \in X}$ is stable under conjugation by elements of $\wt W$;
\item \label{prop:tildesi:permutePhiX+} 
$\Phi_X^+$ is $\wt W$-stable; 
\item \label{prop:tildesi:nobadnodes}
$(X,\tau)$ is a generalized Satake diagram;
\item \label{prop:tildesi:Vsigma}
$\wt W$ stabilizes $V^\si$;
\item \label{prop:tildesi:baralphai}
$\wt s_i(\ol{\al_i}) = -\ol{\al_i}$ for all $i \in \wt I$;
\item \label{prop:tildesi:normalizeX} 
$\wt W \le N_W(W_X)$.
\end{enumerate}
\end{prop}

\begin{proof} 
Note that each statement except \ref{prop:tildesi:nobadnodes} is a statement about the finite group $W_{X[i]}$; for instance, \ref{prop:tildesi:Vsigma} is equivalent to the statement that for all $i \in \wt I$, $\wt s_i$ stabilizes $V^\si$.
Furthermore, recall Definition \ref{def:GSat} \ref{def:GSat:vanilla}.
Since the excluded connected component 
\tp[baseline=-3pt,line width=.7pt,scale=75/100]{
\draw (.1,0) --  (.5,0);
\filldraw[fill=white] (0,0) circle (.1);
\filldraw[fill=black] (.5,0) circle (.1);
} 
is of finite type, it follows that $(X,\tau)$ is a generalized Satake diagram if and only if for all $i \in \wt I$ such that $\tau(i)=i$ the connected component of the subdiagram $X[i]$ containing $i$ is not of type ${\sf A}_2$.
Hence all statements in Proposition \ref{prop:tildeW}, including \ref{prop:tildesi:nobadnodes}, can be written in the form ``For all $i \in \wt I$, ...''.
It therefore suffices to prove the equivalence of the statements for a given fixed $i \in \wt I$.
We can refer to the proof of the finite-type result \cite[Thm.\ 2.1]{RV20} for the equivalence of statements \ref{prop:tildesi:commutesigma}, \ref{prop:tildesi:nobadnodes} and \ref{prop:tildesi:baralphai}, as well as the statement
\eq{ \label{prop:tildesi:eq1}
w_X \cdot w_{X[i]} = w_{X[i]} \cdot w_X,
}
which is easily seen to be equivalent to \ref{prop:tildesi:involution}.
By \eqref{wX:conjugation} (with $X$ replaced by $X[i]$), \eqref{prop:tildesi:eq1} can be rephrased as the condition
\eq{ \label{prop:tildesi:eq2}
\oi_{X[i]}(w_X) = w_X
}
which we encountered above, and which in turn is equivalent to \ref{prop:tildesi:preserveX}.
By \eqref{wX:diagaut}, \eqref{prop:tildesi:eq2} is equivalent to 
\eq{ \label{prop:tildesi:eq3}
\oi_{X[i]}(X) = X.
}
In turn, this is equivalent to  
\eq{ \label{prop:tildesi:eq4}
w_{X[i]}(\Phi^+_X) = -\Phi^+_X.
}
Since also $w_X(\Phi^+_X)=-(\Phi^+_X)$, \eqref{prop:tildesi:eq4} is equivalent to \ref{prop:tildesi:permutePhiX+}.

In order prove that the remaining statements \ref{prop:tildesi:Vsigma} and \ref{prop:tildesi:normalizeX} are equivalent to the others, it remains to establish some implications.
\begin{description}
\item[{\ref{prop:tildesi:commutesigma} $\Rightarrow$ \ref{prop:tildesi:Vsigma}}]
This is clear.
\item[{\ref{prop:tildesi:preserveX} $\Rightarrow$ \ref{prop:tildesi:normalizeX}}]
This is clear.
\item[{\ref{prop:tildesi:Vsigma} $\Rightarrow$ \ref{prop:tildesi:nobadnodes} 
and \ref{prop:tildesi:normalizeX} $\Rightarrow$ \ref{prop:tildesi:nobadnodes}}]
We prove the contrapositives.
Suppose 
\tp[baseline=-3pt,line width=.7pt,scale=75/100]{
\draw (.1,0) --  (.5,0);
\filldraw[fill=white] (0,0) circle (.1);
\filldraw[fill=black] (.5,0) circle (.1);
} 
is a connected component of $X[i]$.
Then $\tau(i)=i$ and there exists $j \in X$ such that $w_{X[i]}|_{\Sp_\Q\{\al_i,\al_j\}}$ equals $s_i s_j s_i = s_j s_i s_j$.
Hence $\wt s_i = w_{\{i,j\}}w_{\{i\}} = s_is_j$.
Note that $\si|_{\Sp_\Q\{\al_i,\al_j\}} = s_j$.
We obtain $\ol{\al_i} = \al_i + \frac{1}{2}\al_j \in V^\si$.
It follows that $\wt s_i(\ol{\al_i}) = s_i(\al_i + \tfrac{1}{2}\al_j) = \tfrac{1}{2}(\al_j-\al_i)$ which is not fixed by $\si$, as required.
Furthermore, conjugation by $w_{X[i]}$ maps $s_j$ to $s_i$, so that $w_{X[i]}$ does not normalize $W_X$. 
\hfill \qedhere
\end{description}
\end{proof}

\begin{rmk}
Suppose $(X,\tau)$ is \emph{not} a generalized Satake diagram.
Let $i \in I^*$ be such that $X[i]$ has
\tp[baseline=-3pt,line width=.7pt,scale=75/100]{
\draw (.1,0) --  (.5,0);
\filldraw[fill=white] (0,0) circle (.1);
\filldraw[fill=black] (.5,0) circle (.1);
} 
as a connected component.
\begin{enumerate}
\item
Note that $\wt s_i$ does not stabilize $V^\si$, but it does stabilize the subspace $\{ \be \in V^\si \, : \, (\be,\ol{\al_i}) = 0 \}$, so there is no contradiction between Lemma \ref{lem:WextendedXfix} and Proposition \ref{prop:tildeW} \ref{prop:tildesi:Vsigma}. 
\item
The generator $\wt s_i$ is of order 3.
Hence in this case $\wt W$ is a (possibly infinite) complex reflection group.
We will not pursue this further here.
\hfill \rmkend
\end{enumerate}
\end{rmk}

%%% 4.8.

\subsection{The group $W(\ol \Phi)$ revisited} \label{sec:W:barPhi}

Note that Lemma \ref{lem:WextendedXfix} and Proposition \ref{prop:tildeW} \ref{prop:tildesi:Vsigma}-\ref{prop:tildesi:baralphai} combine to imply that, for all $i \in \wt I$, 
\eq{ \label{tildesi:barsi}
\wt s_i|_{V_\si} = \ol{s_i}.
}
As a consequence we obtain a group map $\phi: \wt W \to W(\ol{\Phi})$.
We now proceed as in \cite[Sec.\ 2.3]{DK19} to show that $\phi$ is an isomorphism.
Denote by $W'(\ol\Phi)$ the subgroup of $W(\ol{\Phi})$ generated by $\{ \ol{s_i} \}_{i \in \wt I}$.
The following lemma is similar to a key technical step in the standard proof that Weyl groups of root systems are generated by the simple reflections, see e.g.~\cite[Thm.\ 1.5, Step (1) of proof]{Hu90}, but we feel it is beneficial to go through it carefully in the restricted Kac-Moody setting.

\begin{lemma} \label{lem:WbarPhi:simpleroot}
Let $\be \in \ol{\Phi^+}$ be and assume that $(\be,\be)>0$.
Then the $W'(\ol\Phi)$-orbit through $\be$ intersects $\Z_{\ge 0}\ol{\al_i}$ for some $i \in \wt I$.
\end{lemma}

\begin{proof}
Denote the orbit under consideration by $[\be]$.
By Lemma \ref{lem:innerproducts}, if $\ol{\al_i} \in [\be]$ then $i \in \wt I$.
Hence it is sufficient to prove that $[\be]$ intersects $\Z_{\ge 0}\ol{\al_i}$ for some $i \in I^*$.

Since $[\be] \cap \ol{\Phi^+}$ contains $\be$, it is nonempty.
Now pick any $\mu \in \Phi^+$ such that $\ol\mu$ is an element of $[\be] \cap \ol{\Phi^+}$.
By Lemma \ref{lem:Vsigma:base}, $\ol\mu = \sum_{j \in I^*} m_j \ol{\al_j}$ with $m_j \in \Z_{\ge 0}$. 
Since $(\ol\mu,\ol\mu) = (\be,\be) > 0$, there exists $i \in I^*$ such that $(\ol\mu,\ol{\al_i})>0$.
If $\ol\mu \in \Z_{\ge 0} \ol{\al_i}$, we are done.
Otherwise consider $\ol{s_i}(\ol \mu)$ which lies in $\ol{\Phi^+} \backslash \Q\ol{\al_i}$ by Lemma \ref{lem:simpleresrefl:stable}.
We obtain
\eq{ \label{lem:WbarPhi:simpleroot:1}
\ol{s_i}(\ol \mu) - \ol \mu = - 2 \frac{(\ol{\al_i},\ol{\mu})}{(\ol{\al_i},\ol{\al_i})} \ol{\al_i} \in \Q_{< 0} \ol{\al_i}.
}
On the other hand,
\eq{
\ol{s_i}(\ol \mu) - \ol \mu  = \wt s_i(\ol \mu) - \ol \mu = \ol{\wt s_i(\mu)-\mu}
}
as a consequence of \eqref{tildesi:barsi} and Proposition \ref{prop:tildeW} \ref{prop:tildesi:commutesigma}.
Since $\wt s_i(\mu)-\mu \in \Sp_\Z\{ \al_j \}_{j \in X[i]}$ it follows that $\ol{s_i}(\ol \mu) - \ol \mu \in \Sp_\Z\{ \ol{\al_j} \}_{j \in X[i]} = \Z \ol{\al_i}$.
Now \eqref{lem:WbarPhi:simpleroot:1} implies that $\ol{s_i}(\ol \mu) - \ol \mu$ lies in $\Z_{< 0} \ol{\al_i}$, so that $\ol{s_i}(\ol \mu) \in \ol{\Phi^+}$ has smaller restricted height than $\ol \mu$.

We repeat the above steps with $\ol\mu$ replaced by $\ol{s_i}(\ol\mu)$.
This process terminates, since the restricted simple roots are precisely the elements of $\ol{\Phi^+}$ with restricted height 1 so that at some point an element of $\Z_{> 0} \ol{\al_i}$ is obtained for some $i \in I^*$.
\end{proof}

\begin{prop} \label{prop:WbarPhi:generators}
The group $W(\ol{\Phi})$ is generated by $\ol{s_i}$ for $i \in \wt I$.
\end{prop}

\begin{proof}
We need to prove that $W'(\ol\Phi)$ equals $W(\ol\Phi)$.
From Lemma \ref{lem:WbarPhi:simpleroot} it follows that
\eq{
\{ \be \in \ol{\Phi^+} \, : \, (\be,\be)>0 \} 
}
lies in the $W'(\ol\Phi)$-image of $\cup_{i \in \wt I} \Z_{>0} \ol{\al_i}$.
Since elements of $W(\ol\Phi)$ are isometries, by \eqref{restrictedreflection:isometry} we obtain that for all $\be \in \ol{\Phi^+}$ such that $(\be,\be)>0$ there exist $w \in W'(\ol \Phi)$, $i \in \wt I$ and $m \in \Z_{>0}$ such that
\eq{
s_\be = w s_{m \ol{\al_i}} w^{-1}.
}
We complete the proof by observing that $s_{-\be} = s_\be$ and $s_{m \ol{\al_i}} = \ol{s_i}$.
\end{proof}

Proposition \ref{prop:WbarPhi:generators} shows that $W(\ol\Phi)$ is generated by $\{ \wt s_i \}_{i \in \wt I}$ so that the map $\phi$ is surjective.
The following result proves that the map $\phi$ is injective.

\begin{lemma} \label{lem:tildeW:faithful}
If $(X,\tau)$ is a generalized Satake diagram, the action of $\wt W$ on $\ol{\Phi}$ is faithful.
\end{lemma}

\begin{proof}
We give a simplified version of the proof of \cite[Lem.\ 2.8]{DK19}.
Suppose that $w \in \wt W$ fixes $\ol\Phi$ pointwise.
We need to prove that $w = 1$.
In particular, $w(\ol{\al_i})=\ol{\al_i}$ for all $i \in I^*$, i.e.
\eq{ \label{lem:tildeW:faithful:1}
w(\al_i + \si(\al_i)) = \al_i + \si(\al_i) \qq \text{for all } i \in I^*.
}
Recalling that $\si= w_X \circ \tau$ and using Lemma \ref{lem:Xtau:basics} \ref{lem:Xtau:basics:wX:formula} we obtain
\eq{
\si(\al_i) = \al_{\tau(i)} + \sum_{j \in X} v_{ij} \al_j
}
with $v_{ij} \in \Z_{\ge 0}$.
Since the generators of $\wt W$ commute with $\tau$ and act on $\Pi_X:=\{ \al_j \}_{j \in X}$ as a diagram automorphism, 
we deduce from \eqref{lem:tildeW:faithful:1} the relation
\eq{ \label{lem:tildeW:faithful:2}
w(\al_i)  + \tau(w(\al_i)) + \sum_{j \in X} v_{ij} w(\al_j) =  \al_i + \al_{\tau(i)} + \sum_{j \in X} v_{ij} \al_j.
}
Now applying $\ol{\phantom{x}}$ and using \eqrefs{bar:wX}{bar:tau}, we obtain $w(\al_i) = \al_i$ for all $i \in I^*$.
Using once again that $w$ commutes with $\tau$ and acts on $\Pi_X$ as a diagram automorphism, it follows that $w$ stabilizes $\Pi$.
From $\Aut(\Phi) = \Out(A) \ltimes W$ we deduce $w = 1$.
\end{proof}

We can now identify the two groups $\wt W$ and $W(\ol\Phi)$.

\begin{thrm} \label{thm:weylgroups}
$\wt W \cong W(\ol\Phi)$ if and only if $(X,\tau)$ is a generalized Satake diagram.
\end{thrm}

\begin{proof}
If $(X,\tau)$ is a generalized Satake diagram, by Proposition \ref{prop:WbarPhi:generators} and Lemma \ref{lem:tildeW:faithful}, the map $\phi$ is a group isomorphism.
It only remains to prove the ``only if'' part of the first statement, which can be done as before via the contrapositive.
Suppose $(X,\tau)$ is not a generalized Satake diagram.
By Proposition \ref{prop:tildeW}, there exists $i \in \wt I$ and $j \in X$ such that $\wt s_i = s_i s_j$ is of order 3.
It suffices to prove that $\wt W \cap W_{X[i]} = \langle \wt s_i \rangle$ is not isomorphic to $W(\ol{\Phi_{\{i,j\}}})$.
This is clear since by Lemma \ref{lem:Vsigma:base} $W(\ol{\Phi_{\{i,j\}}}) = W(\{ \pm \ol{\al_i}\})$ is a group of two elements.
\end{proof}

As a consequence, we obtain that $\ol \Phi$ is a root system in the following sense:

\begin{crl} \label{WbarPhi:stable}
If $(X,\tau)$ is a generalized Satake diagram, then $\ol \Phi$ is $W(\ol\Phi)$-stable.
\end{crl}

\begin{proof}
Let $i \in \wt I$ be arbitrary.
In view of Theorem \ref{thm:weylgroups} it suffices to prove that $\ol\Phi$ is stable under the action of $\wt s_i|_{V^\si}$.
Note that an arbitrary element of $\ol\Phi$ is of the form $\ol\la$ with $\la \in \Phi \backslash \Phi_X$.
Then $\wt s_i(\ol \la) = \ol{\wt s_i(\la)}$ by Proposition \ref{prop:tildeW} \ref{prop:tildesi:commutesigma}.
Combining $\la \in \Phi \backslash \Phi_X$ with Proposition \ref{prop:tildeW} \ref{prop:tildesi:permutePhiX+}, we obtain $\wt s_i(\la) \in \Phi \backslash \Phi_X$ so that $\ol{\wt s_i(\la)} \in \ol\Phi$ as required.
\end{proof}

The converse to Corollary \ref{WbarPhi:stable} is not true: $\ol \Phi$ is stable under $W(\ol\Phi)$ also if $(X,\tau)$ equals \tp[baseline=-3pt,line width=.7pt, scale=75/100]{
\draw (.1,0) --  (.5,0);
\filldraw[fill=white] (0,0) circle (.1);
\filldraw[fill=black] (.5,0) circle (.1);
} (cf.~\cite[Thm.~6.1]{He84}, which is a result for the group $\ol W$). 

\subsection{The restricted Weyl group as a Coxeter group}

It turns out that the condition that $(\wt W,\{ \wt s_i \})_{i \in \wt I}$ is a Coxeter system is also equivalent to $(X,\tau)$ being a generalized Satake diagram.
We give a brief synopsis of the proof of this in the approaches of \cite[Appendix]{Lu03} and \cite{GI14}.
In particular, we highlight the following property as it identifies where in this approach the condition that $(X,\tau)$ is a generalized Satake diagram is necessary.

\begin{lemma} \label{lem:tildeW:calWtau:1}
$\wt W \leq \mc W^\tau$ if and only if $(X,\tau)$ is a generalized Satake diagram.
\end{lemma}

\begin{proof}
By Proposition \ref{prop:tildeW}, $\wt W \leq N_W(W_X)$ if and only if $(X,\tau)$ is a generalized Satake diagram.
Combining this with \eqref{tildeW:taufixed}, it remains to prove that $\ell(s_j \wt s_i) > \ell(\wt s_i)$ for all $i \in \wt I$ and $j \in X$, assuming that $(X,\tau)$ is a generalized Satake diagram.
For all $j \in X$, $s_j \wt s_i = \wt s_i s_{(\oi_X \oi_{X[i]})(j)}$ with $(\oi_X \oi_{X[i]})(j) \in X$. 
Therefore it suffices to prove that $\ell(\wt s_i w_j) > \ell(\wt s_i)$ for all $j \in X$, which by a basic property of Coxeter groups is equivalent to $\wt s_i$ sending $\{ \al_i \}_{i \in X}$ into $\Phi^+_{X[i]}$.
This follows directly from $\wt s_i(\al_j) = \al_{(\oi_X \oi_{X[i]})(j)}$ for all $j \in X$.
\end{proof}

The inductive proof of \cite[Lem.\ 2]{GI14} can be directly generalized to the current setting, which yields the following statement.

\begin{lemma} \label{lem:tildeW:calWtau:2}
For all $w \in \mc W^\tau$ there exists $(i_1,\ldots,i_{\ell(w)}) \in \wt I^{\ell(w)}$ such that
\eq{
w = \wt s_{i_1} \cdots \wt s_{i_{\ell(w)}}.
}
Furthermore, if $i \in I$ satisfies $\ell(s_iw) < \ell(w)$ then we can choose $(i_1,\ldots,i_{\ell(w)})$ so that $i \in X[i_1]$.
\end{lemma}

We obtain from Lemmas \ref{lem:tildeW:calWtau:1} and \ref{lem:tildeW:calWtau:2} the following result.

\begin{crl} \label{cor:tildeW:calWtau}
$\wt W \cong \mc W^\tau$ if and only if $(X,\tau)$ is a generalized Satake diagram.
\end{crl}

The following desirable property of $\wt W$ can now be obtained.

\begin{thrm} \label{thm:Coxeter}
The group $\wt W$ is a Coxeter group if and only if $(X,\tau)$ is a generalized Satake diagram.
In this case, with $\wt \ell$ denoting the length function for this Coxeter system, we have
\eq{
\qq \wt \ell(ww') = \wt \ell(w)+\wt \ell(w') \qu \Longleftrightarrow \qu \ell(ww')=\ell(w)+\ell(w') \qq \qq \text{for all } w,w' \in \wt W.
}
\end{thrm}

\begin{proof}
We may proceed as in \cite[Lem.\ 4 - Prop.\ 7]{GI14} (also cf.\ \cite{He91} and \cite[Appendix]{Lu03}) by replacing $W^\tau$ by $\mc W^\tau$ according to \cite[Rmk.\ 8]{GI14}.
In this approach, $(\wt W,\{ \wt s_i \}_{i \in \wt I})$ is shown to be a Coxeter system via the characterization that the exchange condition is satisfied, see e.g.\ \cite[Ch.~IV, \S 1.6, Thm.~1]{Bo68}. 
We leave the details to the reader.

To complete the proof, we also need to show that $(\wt W,\{ \wt s_i \}_{i \in \wt I})$ is not a Coxeter system if $(X,\tau)$ is not a generalized Satake diagram, which follows immediately from Proposition \ref{prop:tildeW}.
\end{proof}

Note that Proposition \ref{prop:barW} and Corollary \ref{cor:tildeW:calWtau} combine to yield the following final identification.

\begin{thrm} \label{thm:Weylgroups2}
$\wt W \cong \ol W$ if and only if $(X,\tau)$ is a generalized Satake diagram.
\end{thrm}

\begin{rmk}\mbox{}
\begin{enumerate}
\item
We leave it to the reader to check that the results in this section extend to non-crystallographic Coxeter groups, if we make the following adjustment:
call a compatible decoration $(X,\tau)$ a generalized Satake diagram if for all $i \in I\backslash X$ such that $\tau(i)=i$ the connected component of $X[i]$ containing $i$ is not a nontrivial odd dihedral group, i.e. not of type ${\sf I}_2(m)$ with $m>1$ odd.
In the crystallographic case only ${\sf I}_2(3) = {\sf A}_2$ can occur, which recovers Definition \ref{def:GSat} \ref{def:GSat:vanilla}.
\item In the case $(X,\tau) = \tp[baseline=-3pt,line width=.7pt, scale=75/100]{
\draw (.1,0) --  (.5,0);
\filldraw[fill=white] (0,0) circle (.1);
\filldraw[fill=black] (.5,0) circle (.1);
}$ the three groups $\ol W$, $W(\ol\Phi)$ and $\wt W$ are all distinct: their orders are 1, 2 and 3, respectively. 
\hfill \rmkend
\end{enumerate}
\end{rmk}

%%% 4.10.

\subsection{Non-reduced and non-crystallographic root systems}

In Appendices \ref{sec:tables:finite} and \ref{sec:tables:affine} we also indicate the restricted root system, which can be non-reduced.
In Table \ref{nonreduced} below we recall the standard notation and Dynkin diagrams for non-reduced crystallographic root systems of finite and affine type.
The notation is of the form $(X,Y)_n$ for $n \ge 1$; here $X_n$ is the type of the underlying reduced crystallographic root system which is obtained by deleting $2\be$ for all roots $\be$ such that $(\be,\be)>0$ and $Y_n$ is the type of the underlying reduced crystallographic obtained by deleting all roots $\be$ such that $(\be,\be)>0$ and $2\be$ is also a root.\\

\begin{table}[h]
\caption{
Dynkin diagrams for non-reduced crystallographic finite and affine root systems. 
If and only if a node in such a Dynkin diagram is marked by x, the root system contains in addition to the corresponding simple root $\be$ also the root $2\be$.
}
\label{nonreduced}
\centering
\begin{tabular}{lllll}
Type & Diagram & Constraints & \multicolumn{2}{l}{Special low-rank case} \\
\hline
\hline 
\\[-10pt]
$({\sf B},{\sf C})_n$ &
\tp{
\draw[thick,dashed] (1.5,0) -- (2.5,0);
\draw[double,->] (2.6,0) --  (3,0);
\filldraw[fill=white] (1.5,0) circle (.1);
\filldraw[fill=white] (2.5,0) circle (.1);
\filldraw[fill=white] (3,0) circle (.1) node[below]{\scriptsize x};
}
& $n \ge 1$
&
\tp{
\filldraw[fill=white] (0,0) circle (.1) node[left]{\scriptsize x};
}
&
$n=1$
\\
\hline
\hline
$(\wh{\sf B},\wh{\sf B}^\vee)_n$ &
\tp{
\draw[thick] (-.6,.3) -- (0,0) -- (-.4,-.3);
\draw[thick,dashed] (0,0) -- (1,0);
\draw[double,->] (1,0) --  (1.4,0);
\filldraw[fill=white] (-.6,.3) circle (.1);
\filldraw[fill=white] (-.4,-.3) circle (.1);
\filldraw[fill=white] (0,0) circle (.1);
\filldraw[fill=white] (1,0) circle (.1);
\filldraw[fill=white] (1.5,0) circle (.1) node[below]{\scriptsize x};
}
& $n \ge 2$ 
& 
\tp{
\draw[double,->] (-.5,0) -- (-.1,0);
\draw[double,<-] (.1,0) --  (.5,0);
\filldraw[fill=white] (-.5,0) circle (.1);
\filldraw[fill=white] (0,0) circle (.1) node[below]{\scriptsize x};
\filldraw[fill=white] (.5,0) circle (.1);
}
&
$n=2$
\\
\hline
\\[-10pt]
$(\wh{\sf C}',\wh{\sf C})_n$
&
\tp{
\draw[double,->] (-.5,0) -- (-.1,0);
\draw[thick,dashed] (0,0) -- (1,0);
\draw[double,->] (1,0) --  (1.4,0);
\filldraw[fill=white] (-.5,0) circle (.1);
\filldraw[fill=white] (0,0) circle (.1);
\filldraw[fill=white] (1,0) circle (.1);
\filldraw[fill=white] (1.5,0) circle (.1) node[below]{\scriptsize x};
}
& $n \ge 1$ 
& 
\tp{
\draw[double,-] (-.5,.055) -- (-.2,.055);
\draw[double,-] (-.5,-.055) -- (-.2,-.055);
\draw[double,->] (-.2,0) -- (-.08,0);
\filldraw[fill=white] (-.5,0) circle (.1);
\filldraw[fill=white] (0,0) circle (.1) node[right]{\scriptsize x};
}
& $n=1$ \\
\hline
\\[-10pt]
$(\wh{\sf C}^\vee,\wh{\sf C})_n$
&
\hspace{-7pt}
\tp{
\draw[double,<-] (-.4,0) -- (0,0);
\draw[thick,dashed] (0,0) -- (1,0);
\draw[double,->] (1,0) --  (1.4,0);
\filldraw[fill=white] (-.5,0) circle (.1) node[below]{\scriptsize x};
\filldraw[fill=white] (0,0) circle (.1);
\filldraw[fill=white] (1,0) circle (.1);
\filldraw[fill=white] (1.5,0) circle (.1) node[below]{\scriptsize x};
}
& $n \ge 1$ 
& 
\tp{
\draw[double] (-.5,0) -- (0,0);
\filldraw[fill=white] (-.5,0) circle (.1) node[left]{\scriptsize x};
\filldraw[fill=white] (0,0) circle (.1) node[right]{\scriptsize x};
}
& $n=1$ \\
\hline
\\[-10pt]
$(\wh{\sf C}^\vee,\wh{\sf C}')_n$
&
\hspace{-7pt}
\tp{
\draw[double,<-] (-.4,0) -- (0,0);
\draw[thick,dashed] (0,0) -- (1,0);
\draw[double,->] (1,0) --  (1.4,0);
\filldraw[fill=white] (-.5,0) circle (.1) node[below]{\scriptsize x};
\filldraw[fill=white] (0,0) circle (.1);
\filldraw[fill=white] (1,0) circle (.1);
\filldraw[fill=white] (1.5,0) circle (.1);
}
& $n \ge 1$ 
& 
\tp{
\draw[double] (-.5,0) -- (0,0);
\filldraw[fill=white] (-.5,0) circle (.1) node[left]{\scriptsize x};
\filldraw[fill=white] (0,0) circle (.1);
}
& $n=1$ \\
\hline
\end{tabular}
\end{table}

Only in one finite-type and one affine-type case a generalized Satake diagram yields a non-crystallographic root system (which is also non-reduced).
We consider these in detail now, also as examples of the computation of the restricted root system via subdiagrams of restricted rank 1.

\begin{exam}\mbox{}
\begin{enumerate}
\item 
Suppose $\mfg$ is of type ${\sf G}_2$.
Label the Dynkin diagram as follows: \hspace{-10pt}
\tp[baseline=-5pt,line width=.7pt, scale=70/100]{
\draw[double,double distance=1.7pt] (-.5,0) -- (-.15,0);
\draw[decoration={markings,mark=at position 1 with {\arrow[scale=1.8]{>}}},postaction={decorate}] (-.5,0) -- (-.05,0);
\filldraw[fill=white] (0,0) circle (.1) node[below]{\scriptsize 2};
\filldraw[fill=white] (-.5,0) circle (.1) node[below]{\scriptsize 1};
}
\hspace{-5pt}. 
Let $\si = s_1$.
From
\eq{
\Phi = \pm \{ \al_1, \al_2, \al_1 + \al_2, \al_1 + 2\al_2, \al_1 + 3 \al_2, 2\al_1 + 3\al_2 \}
}
and 
\eq{
\ol{\al_2} = \ol{\al_1 + \al_2} = \tfrac{1}{2}\al_1 + \al_2, \qu \ol{\al_1 + 2\al_2} = \al_1 + 2\al_2, \qu \ol{\al_1 + 3\al_2} = \ol{2\al_1 + 3\al_2} = \tfrac{3}{2} \al_1 + 3 \al_2
}
it follows that $\ol{\Phi} = \pm\{ 1,2,3\} \ol{\al_2}$, a non-reduced non-crystallographic root system of rank~1.
We denote it\footnote{This case is exceptional in another way: the corresponding subalgebra $\mfk$ is the only reductive one with $(X,\tau,\gamma)$ an enriched generalized Satake diagram of finite type such that $(X,\tau)$ is \emph{not} a Satake diagram, see \cite[Ex. 2 (ii)]{RV20}.} by $({\sf B},{\sf C})_1^+$, since we may obtain it by adjoining $3\be$ for all short roots $\be$ to the non-reduced root system of type $({\sf B},{\sf C})_1$.
\item
By adjoining a node corresponding to minus the highest short root to the Dynkin diagram of the previous example we obtain the affine Lie algebra of type $\wh{\sf G}_2^\vee$.
More precisely, we may label the nodes in the Dynkin diagram as follows:
\hspace{-9pt}
\tp[baseline=-5pt,line width=.7pt, scale=70/100]{
\draw[thick] (.5,0) -- (0,0);
\draw[double,double distance=1.7pt] (-.5,0) -- (-.15,0);
\draw[decoration={markings,mark=at position 1 with {\arrow[scale=1.8]{>}}},postaction={decorate}] (-.5,0) -- (-.05,0);
\filldraw[fill=white] (.5,0) circle (.1) node[below]{\scriptsize 0};
\filldraw[fill=white] (0,0) circle (.1) node[below]{\scriptsize 2};
\filldraw[fill=white] (-.5,0) circle (.1) node[below]{\scriptsize 1};
}
\hspace{-5pt}.

Again, let $\si = s_1$. 
As before, $\ol{\al_2} = \frac{1}{2}\al_1 + \al_2$ and straightforwardly we find $\ol{\al_0} = \al_0$.
We may set $\eps_0=\eps_1=1$ and $\eps_2 =3$. 
We compute the inner products:
\eq{
(\ol{\al_0},\ol{\al_0}) = (\al_0,\al_0) = 2, \qu (\ol{\al_0},\ol{\al_2}) = (\al_0,\al_2) = -1, \qu (\ol{\al_2},\ol{\al_2}) = (\al_2,\tfrac{1}{2}\al_1 + \al_2) = \tfrac{1}{2}.
}
Hence $\ol{\al_2}$ and $\ol{\al_0}$ generate a crystallographic root system $\ol\Phi_{\sf red}$ of type $\wh{\sf C}'_1$. 
It is reduced in the sense that $\Z\be \cap \ol\Phi_{\sf red} = \{\pm \be\}$ for all real roots $\be \in \ol\Phi_{\sf red}$.
Integer multiples of $\ol{\al_2}$ arise in this root system: $\ol{\al_1 + 2\al_2} = 2\ol{\al_2}$ and $\ol{2\al_1 + 3\al_2} = 3\ol{\al_2}$.
One verifies that $\ol \Phi$ is obtained from $\ol\Phi_{\sf red}$ by adjoining $2\al$ and $3\al$ for all real short roots $\al$.
Similar to the previous example, we denote this non-reduced non-crystallographic root system by $(\wh{\sf C}',\wh{\sf C})_1^+$. \hfill \examend
\end{enumerate}
\end{exam}

%%%%%%%%%%%%%%%%%%%%%%%%%%%%%%%%%%%%%%%%%%%%%%%%%%%%%%%%%%%%%%%%%%
% EoF
%%%%%%%%%%%%%%%%%%%%%%%%%%%%%%%%%%%%%%%%%%%%%%%%%%%%%%%%%%%%%%%%%%

%%%%%%%%%%%%%%%%%%%%%%%%%%%%%%%%%%%%%%%%%%%%%%%%%%%%%%%%%%%%%%%%%%
% Appendices
%%%%%%%%%%%%%%%%%%%%%%%%%%%%%%%%%%%%%%%%%%%%%%%%%%%%%%%%%%%%%%%%%%

%%%%%%%%%%%%%%%%%%%%%%%%%%%%%%%%%%%%%%%%%%%%%%%%%%%%%%%%%%%%%%%%%%
% Appendix
%%%%%%%%%%%%%%%%%%%%%%%%%%%%%%%%%%%%%%%%%%%%%%%%%%%%%%%%%%%%%%%%%%

\appendix

\setlength{\tabcolsep}{4pt}
\renewcommand*{\arraystretch}{1.2}

\section{Classification of generalized Satake diagrams}  \label{sec:tables}

This Appendix provides a classification of $\Aut(A)$-orbits of generalized Satake diagrams whose underlying Dynkin diagram is of finite or affine type\footnote{For completeness we always include the split and compact Satake diagrams $(\emptyset,\id)$ and $(I,\oi_I)$.}.
We refer the reader to Definition \ref{def:GSat} \ref{def:GSat:vanilla} for the definition of generalized Satake diagram and to \eqref{AutA:actionCDec} for the definition of the $\Aut(A)$-action on the set of generalized Satake diagrams.
For $A$ of finite type we recover the classification obtained in \cite[Table I]{He84}. 
We have also indicated the special $\tau$-orbits (see Section \ref{sec:kprime}) and odd nodes (see Section \ref{sec:cdec}) by marking the corresponding orbits in the diagram ${\sf s}$ and ${\sf o}$, respectively.
The absence of odd nodes in a diagram allows the reader easily to recognize Satake diagrams and to recover the classification obtained in \cite[\S 4 and \S 5]{Ar62} for $A$ of finite type (also see \cite[Sec.~3]{NS95}, \cite[Sec.~7]{Le03} and \cite[Ch.~IX]{He12}) and in \cite[Sec.~6]{BBBR95} for $A$ of affine type.

For each diagram we indicate the type of the restriced root system $\ol{\Phi}$ (see Section \ref{sec:restricted}).
We also give the constraints on the discrete parameters determining the type and the description of $(X,\tau)$ in terms of these parameters. 
Finally, we also enumerate the special cases where the nature of the diagram is different from the general case (e.g.\ due to a different number of special $\tau$-orbits or odd nodes or a low rank case), in order of decreasing restricted rank.

%%% A.1.

\subsection{Notation}

We use the following descriptive notation for the generalized Satake diagrams:
\[
({\rm T})_{\text{description of } X}^{\text{type of } (X,\tau)} .
\]
Here ${\rm T}$ denotes the type of the underlying Dynkin diagram.
The other details are as follows:
\begin{itemize}
\item 
If ${\rm T}$ is of a classical Lie type, the type of $(X,\tau)$ can be plain (indicated by ${\sf pl}$), alternating (${\sf alt}$), reflecting (${\sf rfl}$) or rotational (${\sf rot}$), as explained in Section \ref{sec:GSats}. 
The description of $X$ is a listing of up to two nonnegative integers, except for copies of ${\sf A}_1$ for diagrams of alternating type. 
Each such integer $p$ indicates a connected component of type ${\sf A}_{p-1}$ (for reflecting and rotational types) or ${\sf B}_p$, ${\sf C}_p$ or ${\sf D}_p$ (for plain and alternating types).
\item 
If ${\rm T}$ is of an exceptional Lie type, the type of $(X,\tau)$ amounts to specifying whether $\tau=\id$ (indicated by an absence of a symbol) or $\tau$ is a nontrivial involution (indicated by an apostrophe).
The latter is only relevant if ${\rm T}$ is ${\sf E}_6$, $\wh{\sf E}_6$ or $\wh{\sf E}_7$; in each of these cases there is a unique diagram automorphism of order~2 up to $\Aut(A)$-conjugacy.
The description of $X$ is simply its Lie type. 
When ${\rm T}={\sf G}_2$, it is supplemented by a superscript ${\sf lo}$ or ${\sf sho}$ indicating whether $X$, when it has a single element, corresponds to the long or short simple root of ${\sf G}_2$. 
\end{itemize}

We hope this notation will allow the reader to identify the diagram more easily than the
roman numeral notation commonly used for finite type and the notation in [BBBR95] used for affine type.

%%% A.2.

\subsection{Low-rank coincidences}

We will allow, in the infinite families of classical Lie type, the rank to go below the usual lower bounds with the following natural interpretations:
\[ 
\begin{array}{ll}
\text{Rank} & \text{Diagrams} \\ \hline \\[-1em]
0& {\sf A}_{-1} = {\sf A}_0 = {\sf B}_0 = {\sf C}_0 = {\sf D}_0 = {\sf D}_1 = ({\sf B},{\sf C})_0 = {\sf Z}_0, \\
1& {\sf A}_1 = {\sf B}_1 = {\sf C}_1, \\
2& {\sf A}_1 \times {\sf A}_1 = {\sf D}_2, \qu {\sf B}_2 = {\sf C}_2 \\
3& {\sf A}_3 = {\sf D}_3 
\end{array}
\]
Here ${\sf Z}_0$ denotes the type of the empty set (a finite root system) when considered in isolation (i.e. not as part of an infinite family). 

Similarly, the set of nonzero integer multiples of a fixed nonzero vector is an affine root system, called the \emph{basic imaginary root system}.
We denote the type of this root system by $\wh{\sf Z}_0$ if considered in isolation.
In the infinite families of classical Lie type we will allow the rank to go below the usual lower bounds with the following natural interpretations:
\[
\begin{array}{ll}
\text{Rank} & \text{Diagrams} \\ \hline \\[-1em]
0& \wh{\sf A}_0 = \wh{\sf B}_0 = \wh{\sf B}^\vee_0 = \wh{\sf C}_0 = \wh{\sf C}^\vee_0 = \wh{\sf C}'_0 = (\wh{\sf B},\wh{\sf B}^\vee)_0 = (\wh{\sf C}',\wh{\sf C}^\vee)_0 = (\wh{\sf C}^\vee,\wh{\sf C}^\vee)_0 = (\wh{\sf C}^\vee,\wh{\sf C}')_0 = \wh{\sf Z}_0 \\
1& \wh{\sf A}_1 = \wh{\sf B}_1 = \wh{\sf B}^\vee_1 = \wh{\sf C}_1 = \wh{\sf C}^\vee_1 = (\wh{\sf B},\wh{\sf B}^\vee)_1 \\
2& \wh{\sf B}_2 = \wh{\sf C}_2, \qu \wh{\sf B}^\vee_2 = \wh{\sf C}^\vee_2 \\
3& \wh{\sf A}_3 = \wh{\sf D}_3
\end{array}
\]

%%% A.3.

\subsection{Finite type} \label{sec:tables:finite}

Tables \ref{tab:finite:A}, \ref{tab:finite:BCD} and \ref{tab:finite:ex} list the generalized Satake diagrams.
For each Satake diagram we also give the standard notation in terms of roman numerals.
Note that types ${\sf AIV}$ and ${\sf BDII}$ are distinguished from the larger families ${\sf AIII}$ and ${\sf BDI}$, respectively, by the condition $|I^*| = 1$.

%%% Table 2

\begin{longtable}[h]{m{13mm} m{8mm} m{70mm} m{25mm} m{29mm}}
\caption{
Generalized Satake diagrams of finite Lie type ${\sf A}_n$.
We set $N:=n+1$.
}
\label{tab:finite:A} 
\\
\hline
\multicolumn{2}{l}{Type} & Diagram & Type of $\ol{\Phi}$ & Constraints \\
\hline
\hline
\endfirsthead
\hline
\multicolumn{2}{l}{Type} & Diagram & Type of $\ol{\Phi}$ & Constraints \\
\hline
\endhead
\trimbox{0pt -2pt 0pt -2pt}{$\bigl( {\sf A}_n \bigr)^{\sf pl}$}
& ${\sf AI}$
&
\tp{
\draw[thick] (-.5,0) -- (0,0);
\draw[thick,dashed] (0,0) -- (1,0);
\draw[thick] (1,0) --  (1.5,0);
\filldraw[fill=white] (-.5,0) circle (.1);
\filldraw[fill=white] (0,0) circle (.1);
\filldraw[fill=white] (1,0) circle (.1);
\filldraw[fill=white] (1.5,0) circle (.1);
\filldraw[fill=white,color=white] (2,0) circle (.1);
}
\qu $X=\emptyset$ \hfill $\tau=\id$ \qu
& 
${\sf A}_n$
&
\small $(n > 1)$
\\
\hline
\trimbox{0pt -2pt 0pt -2pt}{$\bigl( {\sf A}_n \bigr)^{\sf alt}$}
& ${\sf AII}$ 
& 
\tp{
\draw[thick] (-.5,0) -- (0,0);
\draw[thick,dashed] (0,0) -- (1,0);
\draw[thick] (1,0) --  (2,0);
\filldraw[fill=black] (-.5,0) circle (.1);
\filldraw[fill=white] (0,0) circle (.1);
\filldraw[fill=black] (1,0) circle (.1);
\filldraw[fill=white] (1.5,0) circle (.1);
\filldraw[fill=black] (2,0) circle (.1);
} \qu $X$ is of type \trimbox{0pt 0pt 0pt -2pt}{${\sf A}_1^{N/2}$} \hfill $\tau=\id$ \qu
& ${\sf A}_{\frac{n-1}{2}}$
& $n$ odd \small $(n > 1)$
\\
\hline 
\hline 
\end{longtable}

%%% Table 3

\enlargethispage{-1em} % this ensures the table is split in the right place

\begin{longtable}[h]{m{15mm} m{10mm} m{70mm} m{30mm} m{22mm}}
\caption{
Generalized Satake diagrams of finite classical Lie type other than ${\sf A}_n$.
}
\label{tab:finite:BCD} 
\\
\hline
\multicolumn{2}{l}{Type} & Diagram & Type of $\ol{\Phi}$ & Constraints \\
\hline
\hline
\endfirsthead
\hline
\multicolumn{2}{l}{Type} & Diagram & Type of $\ol{\Phi}$ & Constraints \\
\hline
\endhead
\trimbox{0pt -2pt 0pt -2pt}{$\bigl( {\sf B}_n \bigr)^{\sf pl}_p$}
& ${\sf BI}$ ${\sf BII}$ 
& 
\tp{
\draw[thick,dashed] (1.5,0) -- (2.5,0);
\draw[thick] (2.5,0) -- (3,0);
\draw[thick,dashed] (3,0) -- (4,0);
\draw[double,->] (4,0) --  (4.4,0);
\filldraw[fill=white] (1.5,0) circle (.1);
\filldraw[fill=white] (2.5,0) circle (.1);
\filldraw[fill=black] (3,0) circle (.1);
\filldraw[fill=black] (4,0) circle (.1);
\filldraw[fill=black] (4.5,0) circle (.1);
}
& ${\sf B}_{n-p}$
& $0 \le p \le n$
\\*
&& \multicolumn{3}{l}{\hspace{40pt} $X$ is of type ${\sf B}_p$ \hspace{56pt}  $\tau=\id$}
\\
\hdashline
%\multicolumn{2}{l}{\multirow{4}{*}{\small Special cases:}}
%& 
%\tp{
%\draw[thick,dashed] (3,0) -- (4,0);
%\draw[double,->] (4,0) --  (4.4,0);
%\filldraw[fill=white] (3,0) circle (.1);
%\filldraw[fill=white] (4,0) circle (.1);
%\filldraw[fill=white] (4.5,0) circle (.1);
%}
%& ${\sf B}_n$
%& $p=0$
%\\*
%\cdashline{3-5}
\multicolumn{2}{l}{\small Special cases:}
& 
\hspace{-6.5pt}
\tp{
\draw[thick] (2,0) -- (3,0);
\draw[thick,dashed] (3,0) -- (4,0);
\draw[double,->] (4,0) --  (4.4,0);
\filldraw[fill=white] (2,0) circle (.1) node[below]{\scriptsize {\sf s}};
\filldraw[fill=white] (2.5,0) circle (.1);
\filldraw[fill=black] (3,0) circle (.1);
\filldraw[fill=black] (4,0) circle (.1);
\filldraw[fill=black] (4.5,0) circle (.1);
}
& ${\sf C}_2$
& $p=n-2$
%\\*
%\cdashline{3-5}
%&
%& 
%\tp{
%\draw[thick] (2.5,0) -- (3,0);
%\draw[thick,dashed] (3,0) -- (4,0);
%\draw[double,->] (4,0) --  (4.4,0);
%\filldraw[fill=white] (2.5,0) circle (.1);
%\filldraw[fill=black] (3,0) circle (.1);
%\filldraw[fill=black] (4,0) circle (.1);
%\filldraw[fill=black] (4.5,0) circle (.1);
%}
%& ${\sf A}_1$
%& $p=n-1$
%\\*
%\cdashline{3-5}
%&
%& 
%\tp{
%\draw[thick,dashed] (3,0) -- (4,0);
%\draw[double,->] (4,0) --  (4.4,0);
%\filldraw[fill=black] (3,0) circle (.1);
%\filldraw[fill=black] (4,0) circle (.1);
%\filldraw[fill=black] (4.5,0) circle (.1);
%}
%& ${\sf Z}_0$
%& $p=n$
\\
\hline
\hline 
\end{longtable}

%%% Table 4

\begin{longtable}[h]{m{13mm} m{10mm} m{23mm} m{25mm} m{12mm} m{9mm} m{23mm} m{18mm}}
\caption{
Generalized Satake diagrams of finite exceptional~Lie~type.
} 
\label{tab:finite:ex} 
\\
\hline
\multicolumn{2}{l}{Type} & Diagram & Type of $\ol{\Phi}$ & \multicolumn{2}{l}{Type} & Diagram & Type of $\ol{\Phi}$ \\
\hline
\hline
\endfirsthead
\hline
\multicolumn{2}{l}{Type} & Diagram & Type of $\ol{\Phi}$ & \multicolumn{2}{l}{Type} & Diagram & Type of $\ol{\Phi}$ \\
\hline
\endhead
\trimbox{0pt -2pt 0pt -2pt}{$\bigl( {\sf E}_6 \bigr)_{\emptyset}$} 
& ${\sf EI}$
&
\tp{
\draw[thick] (1.1,.3) -- (.6,.3) -- (0,0) -- (.4,-.3) -- (.9,-.3);
\draw[thick] (0,0) -- (-.5,0);
\filldraw[fill=white] (1.1,.3) circle (.1);
\filldraw[fill=white] (.9,-.3) circle (.1);
\filldraw[fill=white] (.6,.3) circle (.1);
\filldraw[fill=white] (.4,-.3) circle (.1);
\filldraw[fill=white] (0,0) circle (.1);
\filldraw[fill=white] (-.5,0) circle (.1);
}
& ${\sf E}_6$
& 
\trimbox{0pt -2pt 0pt -2pt}{$\bigl( {\sf E}_6 \bigr)_{{\sf D}_4}$} 
& ${\sf EIV}$
&
\tp{
\draw[thick] (1.1,.3) -- (.6,.3) -- (0,0) -- (.4,-.3) -- (.9,-.3);
\draw[thick] (0,0) -- (-.5,0);
\filldraw[fill=white] (1.1,.3) circle (.1);
\filldraw[fill=white] (.9,-.3) circle (.1);
\filldraw[fill=black] (.6,.3) circle (.1);
\filldraw[fill=black] (.4,-.3) circle (.1);
\filldraw[fill=black] (0,0) circle (.1);
\filldraw[fill=black] (-.5,0) circle (.1);
}
& ${\sf A}_2$
\\
\hline
\hline
\end{longtable}

\newpage

\subsection{Affine type} \label{sec:tables:affine}

Tables \ref{tab:affine:A}, \ref{tab:affine:BCD} and \ref{tab:affine:ex} list the generalized Satake diagrams.

\enlargethispage{1.5em}

%%% Table 5

\begin{longtable}[h]{m{13mm} m{89.5mm} m{23mm} m{29mm}}
\caption{
Generalized Satake diagrams of affine Lie type $\wh{\sf A}_n$. 
We set $N:=n+1$.
} 
\label{tab:affine:A} \\
\hline
Type & Diagram & Type of $\ol{\Phi}$ & Constraints \\
\hline
\hline
\endfirsthead
\hline
Type & Diagram & Type of $\ol{\Phi}$ & Constraints \\
\hline
\endhead
$\bigl( \wh{\sf A}_n \bigr)^{\sf pl}$
& 
\tp[scale=.64,baseline=-4pt]{
\draw[thick,domain=0:135] plot ({cos(\x)},{sin(\x)});
\draw[thick,dashed,domain=135:360] plot ({cos(\x)},{sin(\x)});
\filldraw[fill=white] ({-sqrt(2)/2},{sqrt(2)/2}) circle (.1);
\filldraw[fill=white] (0,1) circle (.1);
\filldraw[fill=white] ({sqrt(2)/2},{sqrt(2)/2}) circle (.1);
\filldraw[fill=white] (1,0) circle (.1);
}
\hspace{5pt}
$X = \emptyset$
\hspace{99pt}
$\tau=\id$
& $\wh{\sf A}_n$
& $\small (n>1)$
\\
\hline
%A.2 
$\bigl( \wh{\sf A}_n \bigr)^{\sf alt}$ 
&
\tp[scale=.64,baseline=-4pt]{
\draw[thick,domain=0:135] plot ({cos(\x)},{sin(\x)});
\draw[thick,dashed,domain=135:360] plot ({cos(\x)},{sin(\x)});
\filldraw[fill=black] ({-sqrt(2)/2},{sqrt(2)/2}) circle (.1);
\filldraw[fill=white] (0,1) circle (.1);
\filldraw[fill=black] ({sqrt(2)/2},{sqrt(2)/2}) circle (.1);
\filldraw[fill=white] (1,0) circle (.1);
}
\hspace{5pt}
$X$ is of type ${\sf A}_1^{N/2}$
\hspace{44.5pt}
$\tau=\id$
& $\wh{\sf A}_{\frac{n-1}{2}}$
& $n$ odd \small $(n>1)$
\\
\hline
%A.4
$\bigl( \wh{\sf A}_n \bigr)^{\sf rot}$ 
&
\hspace{-5pt}
\tp[scale=.64,baseline=-4pt]{
\draw[thick,domain=0:90] plot ({cos(\x)},{sin(\x)});
\draw[thick,dashed,domain=90:180] plot ({cos(\x)},{sin(\x)});
\draw[thick,domain=180:270] plot ({cos(\x)},{sin(\x)});
\draw[thick,dashed,domain=270:360] plot ({cos(\x)},{sin(\x)});
\filldraw[fill=white] (0,1) circle (.1);
\filldraw[fill=white] ({sqrt(2)/2},{sqrt(2)/2}) circle (.1);
\filldraw[fill=white] (1,0) circle (.1);
\filldraw[fill=white] (0,-1) circle (.1);
\filldraw[fill=white] ({-sqrt(2)/2},{-sqrt(2)/2}) circle (.1);
\filldraw[fill=white] (-1,0) circle (.1);
\draw[<->,gray] (0,.9) -- (0,-.9);
\draw[<->,gray] ({.9*sqrt(2)/2},{.9*sqrt(2)/2}) -- ({-.9*sqrt(2)/2},{-.9*sqrt(2)/2});
\draw[<->,gray] (.9,0) -- (-.9,0);
}
\hspace{5pt}
$X = \emptyset$
\hspace{99pt}
$\tau$ = half-turn
& $\wh{\sf A}_{\frac{n-1}{2}}$
& $n$ odd \small $(n>1)$
\\
\hline
\hline
\end{longtable}

\newpage

%%% Table 6

\begin{longtable}{m{14mm} m{83mm} m{26mm} m{31mm}}
\caption{
Generalized Satake diagrams of affine classical Lie type other than $\wh{\sf A}_n$. 
} 
\label{tab:affine:BCD} \\
\hline
Type & Diagram & Type of $\ol{\Phi}$ & Constraints \\
\hline
\hline
\endfirsthead
\hline
Type & Diagram & Type of $\ol{\Phi}$ & Constraints \\
\hline
\endhead
%
%B.1
$\bigl( \wh{\sf B}_n \bigr)^{\sf pl}_{p_1,p_2}$ 
& 
\tp{
\draw[thick] (-.6,.3) -- (0,0) -- (-.4,-.3);
\draw[thick,dashed] (0,0) -- (1,0);
\draw[thick] (1,0) -- (1.5,0);
\draw[thick,dashed] (1.5,0) -- (2.5,0);
\draw[thick] (2.5,0) -- (3,0);
\draw[thick,dashed] (3,0) -- (4,0);
\draw[double,->] (4,0) --  (4.4,0);
\filldraw[fill=black] (-.6,.3) circle (.1);
\filldraw[fill=black] (-.4,-.3) circle (.1);
\filldraw[fill=black] (0,0) circle (.1);
\filldraw[fill=black] (1,0) circle (.1);
\filldraw[fill=white] (1.5,0) circle (.1);
\filldraw[fill=white] (2.5,0) circle (.1);
\filldraw[fill=black] (3,0) circle (.1);
\filldraw[fill=black] (4,0) circle (.1);
\filldraw[fill=black] (4.5,0) circle (.1);
} 
\hspace{2pt} 
\tp{
\draw[thick] (-.5,.3) -- (0,0) -- (-.5,-.3);
\draw[thick,dashed] (0,0) -- (1,0);
\draw[thick] (1,0) -- (1.5,0);
\draw[thick,dashed] (1.5,0) -- (2.5,0);
\draw[thick] (2.5,0) -- (3,0);
\draw[thick,dashed] (3,0) -- (4,0);
\draw[double,->] (4,0) --  (4.4,0);
\draw[<->,gray] (-.5,.2) -- (-.5,-.2);
\filldraw[fill=black] (-.5,.3) circle (.1);
\filldraw[fill=black] (-.5,-.3) circle (.1);
\filldraw[fill=black] (0,0) circle (.1);
\filldraw[fill=black] (1,0) circle (.1);
\filldraw[fill=white] (1.5,0) circle (.1);
\filldraw[fill=white] (2.5,0) circle (.1);
\filldraw[fill=black] (3,0) circle (.1);
\filldraw[fill=black] (4,0) circle (.1);
\filldraw[fill=black] (4.5,0) circle (.1);
}
& $\wh{\sf C}^\vee_{n-p_1-p_2}$
& $p_1,p_2 \ge 0$ \newline $p_1+p_2 \le n$ 
\\*[-4pt]
& \multicolumn{3}{l}{\hspace{20pt}$X$ is of type ${\sf D}_{p_1} \times {\sf B}_{p_2}$ \hspace{62pt} $\tau$ has $\mod(p_1,2)$ nontrivial orbits}
\\
\hline
\hline
\end{longtable}

\newpage

%%% Table 7

\begin{longtable}[h]{llllll}
\caption{
Generalized Satake diagrams of affine exceptional Lie type.
}
\label{tab:affine:ex}\\
\hline 
Type & Diagram & Type of $\ol{\Phi}$ \hspace{1.1cm} & Type & Diagram & Type of $\ol{\Phi}$  \\
\hline
\hline
\endfirsthead
\hline
Type & Diagram & Type of $\ol{\Phi}$ \hspace{1.1cm} & Type & Diagram & Type of $\ol{\Phi}$ \\
\hline
\endhead
$\bigl( \wh{\sf E}_6 \bigr)_{\emptyset}$ 
&
\tp{
\draw[thick] (1.1,.3) -- (.6,.3) -- (0,0) -- (.4,-.3) -- (.9,-.3);
\draw[thick] (0,0) -- (-1,0);
\filldraw[fill=white] (1.1,.3) circle (.1);
\filldraw[fill=white] (.9,-.3) circle (.1);
\filldraw[fill=white] (.6,.3) circle (.1);
\filldraw[fill=white] (.4,-.3) circle (.1);
\filldraw[fill=white] (0,0) circle (.1);
\filldraw[fill=white] (-.5,0) circle (.1);
\filldraw[fill=white] (-1,0) circle (.1);
}
&
$\wh{\sf E}_6$
&
$\bigl( \wh{\sf E}_6 \bigr)_{{\sf D}_4}$ 
&
\tp{
\draw[thick] (1.1,.3) -- (.6,.3) -- (0,0) -- (.4,-.3) -- (.9,-.3);
\draw[thick] (0,0) -- (-1,0);
\filldraw[fill=white] (1.1,.3) circle (.1);
\filldraw[fill=white] (.9,-.3) circle (.1);
\filldraw[fill=black] (.6,.3) circle (.1);
\filldraw[fill=black] (.4,-.3) circle (.1);
\filldraw[fill=black] (0,0) circle (.1);
\filldraw[fill=black] (-.5,0) circle (.1);
\filldraw[fill=white] (-1,0) circle (.1);
}
&
$\wh{\sf A}_2$
\\*
$\bigl( \wh{\sf E}_6 \bigr)_{{\sf A}_1^{ 4}}$ 
&
\tp{
\draw[thick] (1.1,.3) -- (.6,.3) -- (0,0) -- (.4,-.3) -- (.9,-.3);
\draw[thick] (0,0) -- (-1,0);
\filldraw[fill=black] (1.1,.3) circle (.1);
\filldraw[fill=black] (.9,-.3) circle (.1);
\filldraw[fill=white] (.6,.3) circle (.1);
\filldraw[fill=white] (.4,-.3) circle (.1);
\filldraw[fill=black] (0,0) circle (.1);
\filldraw[fill=white] (-.5,0) circle (.1);
\filldraw[fill=black] (-1,0) circle (.1);
}
&
$\wh{\sf A}_2$
\\
\hline
\hline
\end{longtable}

%%%%%%%%%%%%%%%%%%%%%%%%%%%%%%%%%%%%%%%%%%%%%%%%%%%%%%%%%%%%%%%%%%
% EoF
%%%%%%%%%%%%%%%%%%%%%%%%%%%%%%%%%%%%%%%%%%%%%%%%%%%%%%%%%%%%%%%%%%

%%%%%%%%%%%%%%%%%%%%%%%%%%%%%%%%%%%%%%%%%%%%%%%%%%%%%%%%%%%%%%%%%%
% Bibliography
%%%%%%%%%%%%%%%%%%%%%%%%%%%%%%%%%%%%%%%%%%%%%%%%%%%%%%%%%%%%%%%%%%

%%%%%%%%%%%%%%%%%%%%%%%%%%%%%%%%%%%%%%%%%%%%%%%%%%%%%%%%%%%%%%%%%%
% Bibliography
%%%%%%%%%%%%%%%%%%%%%%%%%%%%%%%%%%%%%%%%%%%%%%%%%%%%%%%%%%%%%%%%%%

%%%%%%%%%%%%%%%%%%%%%%%%%%%%%%%%%%%%%%%%%%%%%%%%%%%%%%%%%%%%%%%%%%
% EoF
%%%%%%%%%%%%%%%%%%%%%%%%%%%%%%%%%%%%%%%%%%%%%%%%%%%%%%%%%%%%%%%%%%

\end{document}